\documentclass[3p,times]{elsarticle}
\usepackage{graphicx}
\usepackage{amsmath,amsthm,amsbsy,amssymb,enumerate,latexsym}
\usepackage{calligra,calrsfs,mathrsfs}
\usepackage{multirow}
\usepackage{ulem,cancel}
\usepackage{cases}
\usepackage{bbm}
\usepackage{pgfplots}
\usepackage{subcaption}
\usepackage{tikz}
\usepackage{tkz-euclide}
\usetikzlibrary[patterns]
\usepackage{hyperref}
\usepackage{algorithm}
\usepackage{algpseudocode}
\usepackage{dsfont}
\usepackage{caption}
\usepackage{epstopdf}
\usepackage{epsfig}
\usepackage{array}
\newcolumntype{P}[1]{>{\centering\arraybackslash}p{#1}}
\usepackage{enumerate}
\usepackage{float}
\usepackage[T1]{fontenc}
\usepackage[latin9]{inputenc}
\usepackage{geometry}
\usepgfplotslibrary{fillbetween}
\usetikzlibrary{patterns}
\usepackage{tikz-dimline}
\usepackage{color}
\usepackage{stmaryrd}
\usepackage{setspace}
\usepackage{amsthm}
\usepackage{textcomp}
\usepackage{hyperref}
\usepackage{soul,xcolor}
\usepackage{multirow}
\usepackage{ulem}
\usepackage[inline]{enumitem}
\usepackage{lmodern}
\biboptions{sort&compress}
\allowdisplaybreaks

\newcommand{\assign}{:=}
\newcommand{\nobracket}{}
\newcommand{\nosymbol}{}

\newcommand{\bs}[1]{\ensuremath{\boldsymbol{#1}}}
\newcommand{\tmop}[1]{\ensuremath{\operatorname{#1}}}

\renewenvironment{proof}{\noindent\textbf{Proof\ }}{\hspace*{\fill}$\Box$\medskip}

\newcommand{\ifcomment}{\iffalse}

\newcommand{\ssf}[1]{{\mbox{\sffamily #1}}}
\newcommand{\ssb}[1]{{\mbox{\sffamily \bfseries #1}}}

\newcommand{\avg  }[1]{\langle \, #1 \, \rangle}

\newcommand{\ti}[1]{\tilde{#1}}
\newcommand{\G}{\Gamma}

\newcommand{\Om}{\Omega}
\newcommand{\om}{\omega}

\newcommand{\cT}{{\mathcal T}}
\newcommand{\tO}{\tilde{\Omega}_h}

\newcommand{\tG}{\tilde{\Gamma}_h}

\newcommand{\tGD}{\tilde{\Gamma}_{D,h}}
\newcommand{\tGN}{\tilde{\Gamma}_{N,h}}
\newcommand{\tn}{{\tilde{\bs{n}}}}

\newcommand{\tx}{{\tilde{\bs{x}}}}

\newcommand{\tS}{\ssf{S}_{\bs{\delta}}}

\newtheorem{thm}{Theorem}

\newtheorem{lemma}{Lemma}
\newdefinition{rem}{Remark}
\newtheorem{assumption}{Assumption}

\newcolumntype{M}[1]{>{\centering\arraybackslash}m{#1}}

\newcommand{\logLogSlopeTriangle}[6]
{
	
	\pgfplotsextra
	{
		\pgfkeysgetvalue{/pgfplots/xmin}{\xmin}
		\pgfkeysgetvalue{/pgfplots/xmax}{\xmax}
		\pgfkeysgetvalue{/pgfplots/ymin}{\ymin}
		\pgfkeysgetvalue{/pgfplots/ymax}{\ymax}
		
		\pgfmathsetmacro{\xArel}{#1}
		\pgfmathsetmacro{\yArel}{#3}
		\pgfmathsetmacro{\xBrel}{#1-#2}
		\pgfmathsetmacro{\yBrel}{\yArel}
		\pgfmathsetmacro{\xCrel}{\xArel}
		
		\pgfmathsetmacro{\lnxB}{\xmin*(1-(#1-#2))+\xmax*(#1-#2)} 
		\pgfmathsetmacro{\lnxA}{\xmin*(1-#1)+\xmax*#1} 
		\pgfmathsetmacro{\lnyA}{\ymin*(1-#3)+\ymax*#3} 
		\pgfmathsetmacro{\lnyC}{\lnyA+#4*(\lnxA-\lnxB)}
		\pgfmathsetmacro{\yCrel}{\lnyC-\ymin)/(\ymax-\ymin)} 
		
		\coordinate (A) at (rel axis cs:\xArel,\yArel);
		\coordinate (B) at (rel axis cs:\xBrel,\yBrel);
		\coordinate (C) at (rel axis cs:\xCrel,\yCrel);
		
		\draw[#5,line width = #6 mm]   (A)-- node[pos=0.5,anchor=north] {1}
		(B)-- 
		(C)-- node[pos=0.5,anchor=west] {#4}
		cycle;
	}
}

\newcommand{\ReverseLogLogSlopeTriangle}[6]
{
	
	\pgfplotsextra
	{
		\pgfkeysgetvalue{/pgfplots/xmin}{\xmin}
		\pgfkeysgetvalue{/pgfplots/xmax}{\xmax}
		\pgfkeysgetvalue{/pgfplots/ymin}{\ymin}
		\pgfkeysgetvalue{/pgfplots/ymax}{\ymax}
		
		\pgfmathsetmacro{\xArel}{#1}
		\pgfmathsetmacro{\yArel}{#3}
		\pgfmathsetmacro{\xBrel}{#1-#2}
		\pgfmathsetmacro{\yBrel}{\yArel}
		\pgfmathsetmacro{\xCrel}{\xBrel}
		
		\pgfmathsetmacro{\lnxB}{\xmin*(1-(#1-#2))+\xmax*(#1-#2)} 
		\pgfmathsetmacro{\lnxA}{\xmin*(1-#1)+\xmax*#1} 
		\pgfmathsetmacro{\lnyA}{\ymin*(1-#3)+\ymax*#3} 
		\pgfmathsetmacro{\lnyC}{\lnyA-#4*(\lnxA-\lnxB)}
		\pgfmathsetmacro{\yCrel}{\lnyC-\ymin)/(\ymax-\ymin)} 
		
		\coordinate (A) at (rel axis cs:\xArel,\yArel);
		\coordinate (B) at (rel axis cs:\xBrel,\yBrel);
		\coordinate (C) at (rel axis cs:\xCrel,\yCrel);
		
		\draw[#5,line width = #6 mm]     (A)-- node[pos=0.5,anchor=south] {1}
		(B)-- node[pos=0.5, anchor = east] {#4}	
		(C)-- 
		cycle;
	}
}

\newcommand{\ReverseLogLogSlopeTriangleFlip}[6]
{
	
	\pgfplotsextra
	{
		\pgfkeysgetvalue{/pgfplots/xmin}{\xmin}
		\pgfkeysgetvalue{/pgfplots/xmax}{\xmax}
		\pgfkeysgetvalue{/pgfplots/ymin}{\ymin}
		\pgfkeysgetvalue{/pgfplots/ymax}{\ymax}
		
		\pgfmathsetmacro{\xArel}{#1}
		\pgfmathsetmacro{\yArel}{#3}
		\pgfmathsetmacro{\xBrel}{#1-#2}
		\pgfmathsetmacro{\yBrel}{\yArel}
		\pgfmathsetmacro{\xCrel}{\xBrel}
		
		\pgfmathsetmacro{\lnxB}{\xmin*(1-(#1-#2))+\xmax*(#1-#2)} 
		\pgfmathsetmacro{\lnxA}{\xmin*(1-#1)+\xmax*#1} 
		\pgfmathsetmacro{\lnyA}{\ymin*(1-#3)+\ymax*#3} 
		\pgfmathsetmacro{\lnyC}{\lnyA-#4*(\lnxA-\lnxB)}
		\pgfmathsetmacro{\yCrel}{\lnyC-\ymin)/(\ymax-\ymin)} 
		
		\coordinate (A) at (rel axis cs:\xArel,\yArel);
		\coordinate (B) at (rel axis cs:\xBrel,\yBrel);
		\coordinate (C) at (rel axis cs:\xCrel,\yCrel);
		
		\draw[#5,line width = #6 mm]     (A)-- node[pos=0.5,anchor=north] {1}
		(B)-- node[pos=0.5, anchor = east] {#4}	
		(C)-- 
		cycle;
	}
}

\journal{Computer Methods in Applied Mechanics and Engineering}

\begin{document}
	
\begin{frontmatter}

	\title{A penalty-free Shifted Boundary Method of arbitrary order}

	\author[duke]{J. Haydel Collins}
	\ead{haydel.collins@duke.edu}
	\author[franchecomte]{Alexei Lozinski\corref{ca}}
	\ead{alexei.lozinski@univ-fcomte.fr}
	\author[duke]{Guglielmo Scovazzi\corref{ca}}
	\ead{guglielmo.scovazzi@duke.edu}
	\address[duke]{Department of Civil and Environmental Engineering, Duke University, Durham, North Carolina 27708, USA}
	\address[franchecomte]{Universit\'{e} de Franche-Comt\'{e}, CNRS, LmB, F-25000 Besan\c{c}on, France}
	\cortext[ca]{Alexei Lozinski, Guglielmo Scovazzi}

	\begin{abstract} 
	We introduce and analyze a penalty-free formulation of the Shifted Boundary Method (SBM), inspired by the asymmetric version of the Nitsche method.
	We prove its stability and convergence for arbitrary order finite element interpolation spaces and we test its performance with a number of numerical experiments. 
	Moreover, while the SBM was previously believed to be only asymptotically consistent (in the sense of Galerkin orthogonality), we prove here that it is indeed exactly consistent.
	{\color{olive} This contribution is dedicated to Thomas J.R. Hughes, in honor of his lifetime achievements.}
	\end{abstract}

	\end{frontmatter}

\section{Introduction \label{sec:intro} }

Recently, immersed/embedded/unfitted computational methods have received great attention for problems involving complex geometrical features, described in standard formats (i.e., CAD) and non-standard formats (i.e., STL, level sets, etc.). 
In fact, immersed/embedded/unfitted have the potential to drastically reduce the pre-processing time involved in the acquisition of the geometry and the generation of the computational grid.
An (incomplete) list of recent developments include the Immersed Boundary Finite Element Method (IB-FEM)~\cite{boffi2003finite}, the cutFEM~\cite{burman2010ghost,hansbo2002unfitted,schott2015face,burman2018cut,burman2019dirichlet,burman2017cut,burman2010fictitious,burman2012fictitious,burman2014unfitted,burman2018shape,massing2015nitsche,burman2015cutfem,kamensky2017immersogeometric,xu2016tetrahedral,badia2018aggregated}, the Finite Cell Method~\cite{parvizian2007finite,duster2008finite}, and similar earlier methods~\cite{hollig2003finite,hollig2001weighted,ruberg2012subdivision,ruberg2014fixed,lozinski2016new}.

Most of these approaches require the geometric construction of the partial elements cut by the embedded boundary, which can be both algorithmically complicated and computationally intensive, due to data structures that are considerably more complex with respect to corresponding fitted finite element methods. 
Furthermore, integrating the variational forms on the characteristically irregular cut cells may also be difficult and advanced quadrature formulas might need to be employed~\cite{parvizian2007finite,duster2008finite}. 
 
The Shifted Boundary Method (SBM) was proposed  in~\cite{main2018shifted0} as an alternative to existing embedded/unfitted boundary methods and belongs to the more specific class of approximate domain methods~\cite{bramble1972projection,bramble1996finite,bramble1994robust,cockburn2012solving,cockburn2014priori,cockburn2014solving,bertoluzza2005fat,bertoluzza2011analysis,glowinski1994fictitious,lozinski2016new}.
The $\phi$-FEM~\cite{duprez2020phi,duprez2022immersed,cotin2022varphi,duprez2023phi,duprez2023phi2,duprez2023new} also belongs to this class, although with some key differences.
The SBM was proposed in~\cite{main2018shifted0} for the Poisson and Stokes flow problems and generalized in~\cite{main2018shifted} to the advection-diffusion and Navier-Stokes equations, and in~\cite{song2018shifted} to hyperbolic conservation laws.
An analysis of the stability and accuracy of the SBM for the Poisson, advection-diffusion, and Stokes operators was also included in~\cite{main2018shifted0,main2018shifted,atallah2020analysis}, respectively.  A high-order version of the SBM was proposed in~\cite{atallah2022high}, applications to solid and fracture mechanics problems were presented in~\cite{liu2020shift,atallah2021shifted,li2021shifted,li2023blended,li2021shiftedsimple} and simulations of static and moving interfaces were developed  in~\cite{li2020shifted,colomes2021weighted}.

The SBM is built for minimal computational complexity and does not contain any cut cell by design.
Specifically, the location where boundary conditions are applied is {\it shifted} from the true to an approximate (surrogate) boundary, and, at the same time, modified ({\it shifted}) boundary conditions are applied in order to avoid a reduction in the convergence rates of the overall formulation.
In fact, if the boundary conditions associated to the true domain are not appropriately modified on the surrogate domain, only first-order convergence is to be expected. The shifted boundary conditions are imposed by means of Taylor expansions and are applied weakly, using Nitsche's method, leading to a simple, robust, accurate and efficient algorithm. 

In the original and subsequent works on the SBM, boundary conditions were enforced by means of a Nitsche-like approach, inspired by the symmetric penalty Galerkin formulation. 
This strategy requires the selection of a penalty parameter, which might be tedious to estimate in practical engineering computations. 
In the context of immersed methods, there has been recent interest in developing penalty-free methods~\cite{oyarzua2019high}, and we propose here an alternative SBM inspired by the work of Burman~\cite{burman12}.
We prove stability, consistency and convergence of this new, parameter-free, SBM formulation in the context of the Poisson equation, and we test it in a number of preliminary numerical experiments. 
We also extend these ideas to the case of compressible linear elasticity.

In the derivation of the mathematical proofs, we also introduce an interpretation of the SBM that allow us to prove its {\it exact} Galerkin consistency, while before the method was only thought to be asymptotically consistent.

This article is organized as follows: Section~\ref{sec:sbm_intro} introduces the general SBM notation; the analysis of the penalty-free SBM variational formulation of the Poisson problem and its numerical analysis of stability and convergence is discussed in Section~\ref{sec:sbm_poisson}; the extension of the method to the equations of compressible linear elasticity is discussed in Section~\ref{sec:sbm_elasticity}; a series of numerical tests is presented in Section~\ref{sec:numerical_results}; and, finally, conclusions are summarized in Section~\ref{sec:summary}.

\section{Preliminaries} \label{sec:sbm_intro}

\subsection{The true domain, the surrogate domain and maps}
\label{sec:sbmDef}

Let $\Om$ be a connected open set in $\mathbb{R}^{d}$ with Lipschitz boundary $\G = \partial \Om$, and let $\bs{n}$ denote the outer-pointing normal to $\G$.
We consider a closed domain ${\cal D}$ such that $\text{clos}(\Om) \subseteq {\cal D}$, and we introduce a family $\cT_h$ of admissible, shape-regular, and quasi-uniform triangulations of ${\cal D}$.
We will indicate by $h_T$ the size of element $T \in \cT_h$ and by $h$ the piecewise constant function such that $h_{|T}=h_T$.

\begin{rem}
The assumption of quasi-uniformity is not essential for the numerical analysis of the proposed methods, but it greatly simplifies the notation.
\end{rem}
We restrict each triangulation by selecting those elements that are contained in $\text{clos}(\Om)$, i.e., we form
$$
\ti{\cT}_h := \{ T \in \cT_h : T \subset \text{clos}(\Om) \} \, ,
$$ 
which identifies the {\sl surrogate domain}
$$
\tO := \text{int} \left(\bigcup_{T \in \ti{\cT}_h}  T \right) \subseteq \Om \,,
$$
with {\sl surrogate boundary} $\tG:=\partial \tO$ and outward-oriented unit normal vector $\ti{\bs{n}}$ to $\tG$. 
Obviously, $\ti{\cT}_h$ is an admissible, shape-regular triangulation of $\tO$ (see Figure~\ref{fig:SBM}).
\begin{figure}
	\centering
	\begin{subfigure}[b]{.4\textwidth}\centering
		\begin{tikzpicture}[scale=0.85]
		\draw [black, draw=none,name path=surr] plot coordinates { (-2,-3.4641) (-1,-1.73205) (0,-0.5) (1,1.73205) (2.6,2.2) (5,1.73205) (7,2.1) (7.62,0.8) };
		\draw [blue, name path=true] plot[smooth] coordinates {(-0.7,-3.4641) (1.75,0.75) (8.25,0.75)};
		\tikzfillbetween[of=true and surr,split]{gray!15!};
		\draw[line width = 0.25mm,densely dashed,gray] (-1,1.73205) -- (0,3.4641);
		\draw[line width = 0.25mm,densely dashed,gray] (0,3.4641) -- (2,0);
		\draw[line width = 0.25mm,densely dashed,gray] (1,1.73205) -- (1,3.4641);
		\draw[line width = 0.25mm,densely dashed,gray] (1,3.4641) -- (0,3.4641);
		\draw[line width = 0.25mm,densely dashed,gray] (1,3.4641) -- (2.6,2.2);
		\draw[line width = 0.25mm,densely dashed,gray] (1,3.4641) -- (3.25,3.4641);
		\draw[line width = 0.25mm,densely dashed,gray] (3.25,3.4641) -- (2.6,2.2);
		\draw[line width = 0.25mm,densely dashed,gray] (3.25,3.4641) -- (5,1.73205);
		\draw[line width = 0.25mm,densely dashed,gray] (3.25,3.4641) -- (6,3.4641);
		\draw[line width = 0.25mm,densely dashed,gray] (6,3.4641) -- (5,1.73205);
		\draw[line width = 0.25mm,densely dashed,gray] (6,3.4641) -- (7,2.1);
		\draw[line width = 0.25mm,densely dashed,gray] (0,-0.5) -- (-2,0.5);
		\draw[line width = 0.25mm,densely dashed,gray] (-2,0.5) -- (-1,1.73205);
		\draw[line width = 0.25mm,densely dashed,gray] (-2,0.5) -- (-1,-1.73205);
		\draw[line width = 0.25mm,densely dashed,gray] (-2,0.5) -- (-2.5,-2);
		\draw[line width = 0.25mm,densely dashed,gray] (-2.5,-2) -- (-1,-1.73205);
		\draw[line width = 0.25mm,densely dashed,gray] (-2.5,-2) -- (-2,-3.4641);
		\draw[line width = 0.25mm,densely dashed,gray] (0,-0.5) -- (-1,1.73205);
		\draw[line width = 0.25mm,densely dashed,gray] (-1,1.73205) -- (1,1.73205);
		\draw[line width = 0.25mm,densely dashed,gray] (0,-0.5) -- (2,0);
		\draw[line width = 0.25mm,densely dashed,gray] (2,0) -- (1,1.73205);
		\draw[line width = 0.25mm,densely dashed,gray] (1,1.73205) -- (0,-0.5);
		\draw[line width = 0.25mm,densely dashed,gray] (2,0) -- (2.6,2.2);
		\draw[line width = 0.25mm,densely dashed,gray] (2.6,2.2) -- (1,1.73205);
		\draw[line width = 0.25mm,densely dashed,gray] (2,0) -- (4,0);
		\draw[line width = 0.25mm,densely dashed,gray] (4,0) -- (2.6,2.2);
		\draw[line width = 0.25mm,densely dashed,gray] (4,0) -- (2.6,2.2);
		\draw[line width = 0.25mm,densely dashed,gray] (2.6,2.2) -- (5,1.73205);
		\draw[line width = 0.25mm,densely dashed,gray] (5,1.73205) -- (4,0);
		\draw[line width = 0.25mm,densely dashed,gray] (4,0) -- (6,0);
		\draw[line width = 0.25mm,densely dashed,gray] (6,0) -- (5,1.73205);
		\draw[line width = 0.25mm,densely dashed,gray] (6,0) -- (7,2.1);
		\draw[line width = 0.25mm,densely dashed,gray] (7,2.1) -- (5,1.73205);
		\draw[line width = 0.25mm,densely dashed,gray] (6,0) -- (8,0);
		\draw[line width = 0.25mm,densely dashed,gray] (8,0) -- (7,2.1);
		\draw[line width = 0.25mm,densely dashed,gray] (0,-0.5) -- (-1,-1.73205);
		\draw[line width = 0.25mm,densely dashed,gray] (-1,-1.73205) -- (1,-1.73205);
		\draw[line width = 0.25mm,densely dashed,gray] (2,0) -- (1,-1.73205);
		\draw[line width = 0.25mm,densely dashed,gray] (1,-1.73205) -- (0,-0.5);
		\draw[line width = 0.25mm,densely dashed,gray] (2,0) -- (3,-1.73205);
		\draw[line width = 0.25mm,densely dashed,gray] (3,-1.73205) -- (1,-1.73205);
		\draw[line width = 0.25mm,densely dashed,gray] (4,0) -- (3,-1.73205);
		\draw[line width = 0.25mm,densely dashed,gray] (2,0) -- (4,0);
		\draw[line width = 0.25mm,densely dashed,gray] (4,0) -- (3,-1.73205);
		\draw[line width = 0.25mm,densely dashed,gray] (3,-1.73205) -- (5,-1.73205);
		\draw[line width = 0.25mm,densely dashed,gray] (5,-1.73205) -- (4,0);
		\draw[line width = 0.25mm,densely dashed,gray] (4,0) -- (6,0);
		\draw[line width = 0.25mm,densely dashed,gray] (6,0) -- (5,-1.73205);
		\draw[line width = 0.25mm,densely dashed,gray] (6,0) -- (7,-1.73205);
		\draw[line width = 0.25mm,densely dashed,gray] (7,-1.73205) -- (5,-1.73205);
		\draw[line width = 0.25mm,densely dashed,gray] (6,0) -- (8,0);
		\draw[line width = 0.25mm,densely dashed,gray] (8,0) -- (7,-1.73205);
		\draw[line width = 0.25mm,densely dashed,gray] (0,-3.4641) -- (-2,-3.4641);
		\draw[line width = 0.25mm,densely dashed,gray] (-2,-3.4641) -- (-1,-1.73205);
		\draw[line width = 0.25mm,densely dashed,gray]  (-1,-1.73205) -- (0,-3.4641);
		\draw[line width = 0.25mm,densely dashed,gray] (0,-3.4641) -- (1,-1.73205);
		\draw[line width = 0.25mm,densely dashed,gray] (0,-3.4641) -- (2,-3.4641);
		\draw[line width = 0.25mm,densely dashed,gray] (2,-3.4641) -- (1,-1.73205);
		\draw[line width = 0.25mm,densely dashed,gray] (2,-3.4641) -- (3,-1.73205);
		\draw [line width = 0.5mm,blue, name path=true] plot[smooth] coordinates {(-0.75,-3.681818) (1.75,0.75) (8.25,0.75)};
		\draw[line width = 0.5mm,red] (1,1.73205) -- (2.6,2.2);
		\draw[line width = 0.5mm,red] (2.6,2.2) -- (5,1.73205);
		\draw[line width = 0.5mm,red] (5,1.73205) --  (7,2.1);
		\draw[line width = 0.5mm,red] (1,1.73205) -- (0,-0.5);
		\draw[line width = 0.5mm,red] (0,-0.5) -- (-1,-1.73205);
		\draw[line width = 0.5mm,red] (-1,-1.73205) -- (-2,-3.4641);
		\node[text width=0.5cm] at (7.5,2.1) {\large${\color{red}\tG}$};
		\node[text width=3cm] at (1.25,1.25) {\large${\color{red}\tO}$};
		\node[text width=3cm] at (0.25,-3) {\large${\color{blue}\Om}$};
		\node[text width=0.5cm] at (8.75,0.75) {\large${\color{blue}\G}$};
		\node[text width=3cm] at (4.65,1.5) {\large$\Om \setminus \tO $};
		\node[text width=3cm] at (5.5,3.75) {\large$\tO  \subset \Om $};
		\end{tikzpicture}
		\caption{The true domain $\Om$, the surrogate domain $\tO \subset \Om$ and their boundaries $\ti{\G}_{h}$ and $\G$.}
		\label{fig:SBM}
	\end{subfigure}
	\hspace{2.5cm}
	\begin{subfigure}[b]{.4\textwidth}\centering
		\begin{tikzpicture}
		\draw[line width = 0.25mm,densely dashed,gray] (0,0.5) -- (-1.5,3);
		\draw[line width = 0.25mm,densely dashed,gray] (-1.5,3) -- (0.5,5);
		\draw[line width = 0.25mm,densely dashed,gray] (0,0.5) -- (2.5,2);
		\draw[line width = 0.25mm,densely dashed,gray] (2.5,2) -- (0.5,5);
		\draw[line width = 0.25mm,densely dashed,gray] (0.5,5) -- (0,0.5);
		\draw [line width = 0.5mm,blue, name path=true] plot[smooth] coordinates {(1,-0.5) (2.25,2.5) (0.75,6)};
		\draw[line width = 0.5mm,red] (0,0.5) -- (0.5,5);
		\node[text width=0.5cm] at (0.5,5.5) {\large${\color{red}\tG}$};
		\node[text width=0.5cm] at (1.75,5.5) {\large${\color{blue}\G}$};
		\node[text width=0.5cm] at (1.25,3.25) {\large$\bs{\delta}$};
		\node[text width=0.5cm] at (3,3.5) {\large$\bs{n}$};
		\node[text width=0.5cm] at (2.7,2.25) {\large$\bs{\tau}$};
		\draw[-stealth,line width = 0.25mm,-latex] (0.25,2.75) -- (2.12,3.1);
		\draw[-stealth,line width = 0.25mm,-latex] (2.12,3.1) -- (2.35,2.1);
		\draw[-stealth,line width = 0.25mm,-latex] (2.12,3.1) -- (2.95,3.25);
		\end{tikzpicture}
		\caption{The distance vector $\bs{\delta}$, the true normal $\bs{n}$ and the true tangent $\bs{\tau}$.}
		\label{fig:ntd}
	\end{subfigure}
	\caption{The surrogate domain, its boundary, and the distance vector $\bs{\delta}$.}
	\label{fig:surrogates}
\end{figure}
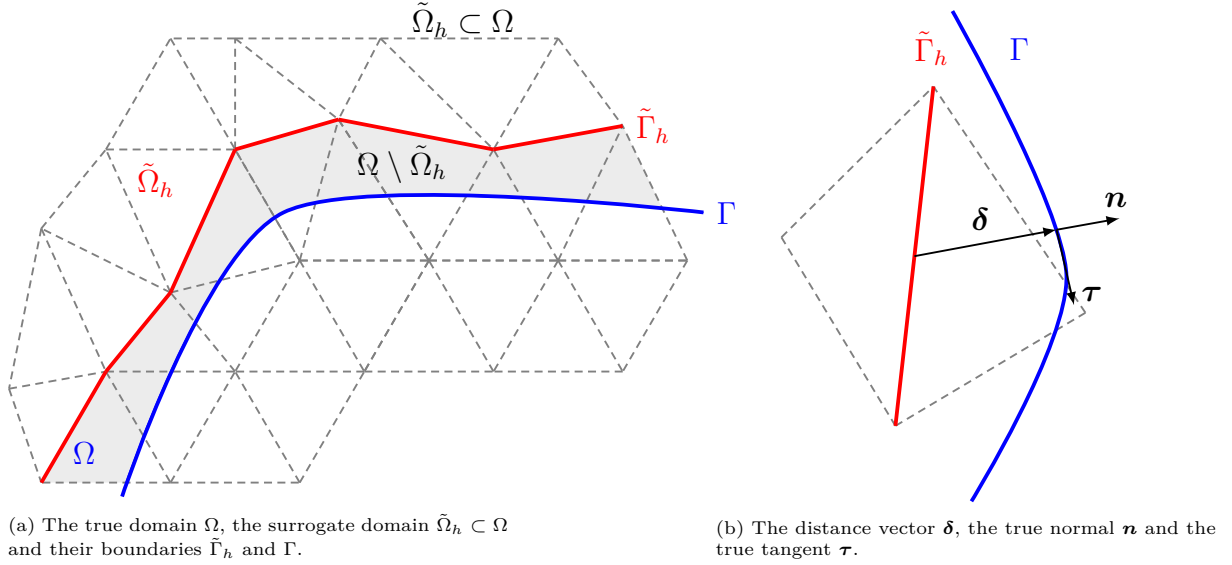
We now introduce a mapping
\begin{subequations}\label{eq:defMmap}
	\begin{align}
	\bs{M}_{h}:&\; \tG \to \G \; ,  \\
	&\; \ti{\bs{x}} \mapsto \bs{x}   \; ,
	\end{align}
\end{subequations}
which associates to any point $\ti{\bs{x}} \in \tG$ on the surrogate boundary a point $\bs{x} = \bs{M}_{h}(\ti{\bs{x}})$ on the physical boundary $\Gamma$.  Whenever uniquely defined, the closest-point projection of $\ti{\bs{x}}$ upon $\Gamma$ is a natural choice for $\bs{x}$, as shown e.g. in Figure~\ref{fig:ntd}. 
Through $\bs{M}_{h}$, a distance vector function $\bs{\delta}_{\bs{M}_{h}}$ can be defined as
\begin{align}
\label{eq:Mmap}
\bs{\delta}_{\bs{M}_{h}} (\ti{\bs{x}})
\, = \, 
\bs{x}-\ti{\bs{x}}
\, = \, 
[ \, \bs{M}_{h}-\bs{I} \, ] (\ti{\bs{x}})
\; .
\end{align}
For the sake of simplicity, we set $\bs{\delta} = \bs{\delta}_{\bs{M}_{h}} $ where $\bs{\delta} = \|\bs{\delta}\| \bs{\nu}$  and $\bs{\nu}$ is a unit vector. 
\begin{rem}
	If $\bs{x} = \bs{M}_{h}(\ti{\bs{x}})$ does not belong to corners or edges, then the closest-point projection implies $\bs{\nu}=\bs{n}$, where $\bs{n}$ was defined as the outward pointing normal to $\G$.
More sophisticated choices may be locally preferable in the presence of corners or edges and we refer to \cite{atallah2021analysis} for more details.
\end{rem}
\begin{rem}
There are strategies for the definition of the map $\bs{M}_{h}$ and distance $\bs{\delta}$ other than the closest-point projection, such as level sets, for which $\bs{\delta}$ is defined by means of a distance function.
\end{rem}
In case the boundary $\G$ is partitioned into a Dirichlet boundary $\G_{D}$ and a Neumann boundary $\G_{N}$ with $\G = \overline{\G_{D} \cup \G_{N}}$ and $\G_{D} \cap \G_{N} = \emptyset$, we need to identify whether a surrogate edge $\ti{e} \subset \tG$ is associated with $\G_{D}$ or $\G_{N}$. To that end, we partition $\tG$ as $\overline{\tGD \cup \tGN}$ with $\tGD \cap \tGN = \emptyset$ using again a map $\bs{M}_{h}$, such that
\begin{align}\label{def:surrogateDir}
\tGD = \{ \ti{e} \subseteq \tG : \bs{M}_{h}(\ti{e}) \, \subseteq \, \G_{D} \}
\end{align}
and $\tGN = \tG \setminus \tGD$. 


\subsection{General strategy} 
In the SBM, the governing equations are discretized in $\tO$ rather than in $\Om$, with the challenge of accurately imposing boundary conditions on $\tG$.
To this end, boundary conditions are {\it shifted} from $\G$ to $\tG$, by performing the $m$th-order Taylor expansion of the variable of interest at the surrogate boundary. 

Under the assumption that $u$ is sufficiently smooth in the strip between $\tG$  and $\G$, so as to admit a $m$th-order Taylor expansion point-wise, let $\ssf{D}^{i}_{\bs{\delta}}$ denote the $i$th-order directional derivative along $\bs{\delta}$:
$$
\ssf{D}^{i}_{\bs{\delta}} u = \displaystyle{\sum_{\bs{\alpha} \in \mathbb{N}^n, |\bs{\alpha}|=i} \frac{i!}{\bs{\alpha}!}   \frac{\partial^i u}{\partial \bs{x}^{\bs{\alpha}}} \bs{\delta}^{\bs{\alpha}}  } \; . 
$$
Then,  we can write
\begin{align}
u(\bs{x}) = u(\ti{\bs{x}}+\bs{\delta}(\ti{\bs{x}}))=u(\tx) +  \sum_{i = 1}^{m}  \frac{\ssf{D}^{i}_{\bs{\delta}} \, u(\ti{\bs{x}})}{i!} + (\ssf{R}^{m}(u,\bs{\delta}))(\tx)\,, 
\end{align}
where the remainder $\ssf{R}^{m}(u,\bs{\delta})$ satisfies  $|\ssf{R}^{m}(u,\bs{\delta})| = o(\Vert \bs{\delta}\Vert^m)$ as $\Vert \bs{\delta}\Vert \to 0$.
Assume that the Dirichlet condition $u(\bs{x})=u_D(\bs{x})$ needs to be imposed on the true boundary $\G$.
Using the map $\bs{M}_h$, one can extend $u_D$ from $\G$ to $\tG$ as $\bar{u}_D(\tx)=u_D(\bs{M}_h (\tx))$. Then, the Taylor expansion can be used to enforce the Dirichlet condition on $\tG$ rather than $\G$, as 
\begin{equation}\label{eq:trace-u}
\tS^{m} u  - \bar{u}_D + \ssf{R}^{m}(u,\bs{\delta}) = 0 
\; , \qquad \mbox{on } \tG \; ,
\end{equation}
where we have introduced the boundary {\it shift} operator for every $\ti{\bs{x}} \in \tG$, namely:
\begin{equation}\label{eq:def-bndS}
\tS^{m} u(\ti{\bs{x}}) := u(\ti{\bs{x}}) + \sum_{i = 1}^{m}  \frac{\ssf{D}^{i}_{\bs{\delta}} \, u (\ti{\bs{x}})}{i!}  \; .
\end{equation}
Neglecting the remainder $\ssf{R}^{m}(u,\bs{\delta})$, we obtain the final expression of the {\it shifted} approximation of order $m$ of the boundary condition
\begin{equation}\label{eq:Finalu-g}
\tS^{m} u \approx \bar{u}_D  \, , \quad  \mbox{on } \tG \; .
\end{equation}
This shifted boundary condition will be enforced weakly in what follows, and whenever there is no source of confusion, the bar symbol will be removed from the extended quantities, and we would write $u_D$ in place of $\bar{u}_D$.

\subsection{General notation}

Throughout the paper, we will denote the space of square integrable functions on $\Om$ as $L^{2}(\Om)$.
 We will use the Sobolev spaces $H^m(\Om)=W^{m,2}(\Om)$ of index of regularity $m \geq 0$ and index of summability 2, equipped with the (scaled) norm
\begin{equation}
\|v \|_{H^{m}(\Om)} 
= \left( \| \, v \, \|^2_{L^2(\Om)} + \sum_{k = 1}^{m} \| \, l(\Om)^k \ssf{D}^k v \, \|^2_{L^2(\Om)} \right)^{1/2} \; ,
\end{equation}
where $\ssf{D}^k$ is the $k$th-order spatial derivative operator and $l(A)=\mathrm{meas}_{d}(A)^{1/d}$ is a characteristic length of the domain $A$ ($d=2,3$ indicates the number of spatial dimensions). Note that $H^0(\Om)=L^{2}(\Om)$.  As usual, we use a simplified notation for norms and semi-norms, i.e., we set $\| \, v  \, \|_{m,\Om}=\|\, v \, \|_{H^m(\Om)}$ and $| \, v \, |_{k,\Om}= 
\| \,\ssf{D}^k v \,\|_{0,\Om}= \| \,\ssf{D}^k v \, \|_{L^2(\Om)}$.

We also introduce the definition of the $L^{2}$-inner product over $\om \in \Om$, namely $( \, u \, , \, v  \, )_{\om} = \int_{\om} u \, v$, and an analogous inner product on the portion of the boundary $\gamma \subset \partial \Om$, namely $\avg{ u \, , \, w }_{\gamma} = \int_{\gamma} u \, w$.
We can also restrict to $\omega$ and $\gamma$ the norms and seminorms initially defined on $\Omega$ and $\Gamma$, that is $\| \cdot \|_{k,\omega}$, $| \cdot |_{k,\omega}$ and $\| \cdot \|_{0,\gamma}$, for example.

\section{A penalty-free Shifted Boundary Method for the Poisson equation \label{sec:sbm_poisson}}
Consider the Poisson problem with  non-homogeneous Dirichlet boundary conditions:
\begin{subequations}
	\label{eq:SteadyPoisson}
\begin{align}
- \Delta u &=\; f  \qquad \text{\ \ in \ } \Om \; ,
 \\
u &=\;  u_D \qquad \text{on \ }  \G = \partial\Om\; ,
\end{align}
\end{subequations}
where $u$ is the primary variable, $u_D$ its value on the boundary $\G$ and $f$ a body force (i.e., non-homogeneous boundary conditions are enforced on the entire boundary $\G$).

Let us introduce the discrete space 
\begin{align}
V^k_h(\tO)  = \; \left\{ v_h \in C^0(\tO)  \ | \ {v_h}_{|\ti{T}} \in \mathcal{P}^k(\ti{T})  \, , \, \forall \ti{T} \in \ti{\mathcal{T}}_h \right\}  \, ,
\end{align}
where $\mathcal{P}^k(\ti{T})$ is the the space of polynomials of at most order $k$ over the triangle $\ti{T}$.
We propose a penalty-free Shifted Boundary Method, inspired by the antisymmetric  version of the Nitsche's method~\cite{nitscheweak,arnold2002unified} that was analyzed by Burman {\cite{burman12}}. 
Hence, the penalty-free variational statement of~\eqref{eq:SteadyPoisson} can be cast as 
\begin{quote}
Find $u_h \in V^k_h(\tO)$ such that, $\forall w_h \in V^k_h(\tO)$
\begin{subequations}
\label{eq:SB_Poisson_uns}
\begin{align}
a^k_h(u_h \, , \, w_h) 
&=\; 
l_h(w_h) 
\; ,
\end{align}
where
\begin{align}
\label{eq:UnsymmetricShiftedNitscheBilinearForm}
a^k_h(u_h \, , \, w_h) 
&=\; ( \, \nabla u_h \, , \, \nabla w_h  \, )_{\tO} 
- \avg{ \nabla u_h  \cdot \ti{\bs{n}} \, , \, w_h }_{\tG}
+  \avg{ \tS^k u_h \, , \, \nabla w_h \cdot \ti{\bs{n}} }_{\tG}
\; ,
\\[.2cm]
\label{eq:UnsymmetricShiftedNitscheRHS}
l_h(w_h) 
&=\; 
( \, f \, , \, w_h \, )_{\tO} 
+ \avg{ \bar{u}_D \, , \, \nabla w_h \cdot \ti{\bs{n}} }_{\tG}
\; .
\end{align}
\end{subequations}
\end{quote}

\subsection{Main theoretical results}
We introduce the natural $h$-dependent norm
\[ \|u_{\nosymbol} \|_h = \left( |u|^2_{1, \tO} + \| h^{-1/2} \, u\|_{0,
   \tG}^2 \right)^{\frac{1}{2}} \]
and prove stability by way of an inf-sup condition, under the following assumptions:
\begin{assumption}
\label{assu:1}
Denoting $\delta (\ti{\bs{x}}) = |\bs{\delta}(\ti{\bs{x}}) |$ the Euclidean norm of $\bs{\delta}$ at $\ti{\bs{x}} \in \tG$, we assume that
\begin{align}
\label{eq:A1}
\max_{\ti{\bs{x}} \in \tG} \delta (\ti{\bs{x}}) \leqslant C_{\max} h \; .
\end{align}
\end{assumption}
\begin{assumption}
\label{assu:2}
Let $C_{I,k}$ be the constant in the trace inverse inequality $\|\tn \cdot \nabla u_h \|_{0, \ti{e}} \leqslant C_{I,k}  \, | h_{\ti{T}}^{1/2} \, u_h |_{1, \ti{T}}$, for any $u_h \in V^k_h(\tO)$ and for any side $\ti{e} \subset \partial \ti{T}$ of the element $\ti{T} \in \ti{\cT}_h$ (i.e., an edge/face in two/three dimensions).
   We assume that $\exists \beta \in [0, 1)$ such that
\begin{align}
\label{eq:A2}
\| \tS^k u_h - u_h \|_{0, \tG} \leqslant \frac{\beta}{C_{I,k}} \, | h^{1/2} \, u_h |_{1, \tO} \; .
\end{align}
\end{assumption}
\begin{assumption}
\label{assu:3}
   In two dimensions, let $\ti{T} \in \ti{\cT}_h$ be a triangle of the mesh with a side $\ti{e} \subset \tG$. 
   We call $\ti{T}$ a ``normal boundary triangle'' if one of its vertices lies in $\tO \setminus \tG$.
   Otherwise (i.e., if all the 3 vertices of the triangle are on $\tG$) we call $\ti{T}$ an ``abnormal boundary triangle.'' 
   We shall suppose that every abnormal boundary triangle lies at a distance $d \leqslant c_{\ti{T}} h$ from a normal boundary triangle, where the constant $c_{\ti{T}}$ is not large.
   The three-dimensional case is analogous.
\end{assumption}
\begin{assumption}
	\label{ass:NeumannAssumption}
	The Neumann boundary is body-fitted, that is $\G_N = \tGN$.
\end{assumption}
\begin{rem}[validity of Assumption~\ref{assu:1} and~\ref{assu:2}]
Assumption~\ref{assu:1} is normally satisfied with a constant $C_{\max}$ of order 1. 
Assumption~\ref{assu:2} is inspired by a similar assumption made in~\cite{atallah2021analysis} but is more difficult to satisfy in practice and should be considered as a technical condition for the proofs.
In fact, this condition may not be satisfied by the computational grids used in the numerical experiments in Section~\ref{sec:numerical_results}, despite obtaining optimal convergence rates in all simulations.
Note that the $h^{1/2}$ factor appears naturally in Assumption~\ref{assu:2}, since on each side $\ti{e} \subset \partial \ti{T}$ we have 
$$\| \tS^k u_h - u_h \|_{0, \ti{e}} \sim h_{\ti{T}} \, | \nabla u_h | \sqrt{|\ti{e}|} \sim \sqrt{h_{\ti{T}}} \,  | \nabla u_h | \sqrt{|\ti{T}|} \sim \sqrt{h_{\ti{T}}} \,  \| \nabla u_h \|_{0, \ti{T}} \; . $$
The practical difficulty in this assumption consists in requiring that $\beta < 1$.
\end{rem}
\begin{rem}[validity of Assumption~\ref{assu:3}]
For practical purposes, Assumption~\ref{assu:3} can be safely accepted.
While Assumption~\ref{assu:3} could be violated in principle, this happens very rarely in practice.

In fact, the way elements are connected in regular grids prevents the edges/faces associated to several abnormal elements to form a chain along the surrogate boundary. 
Even in the most complex three-dimensional geometries, a chain of abnormal boundary tetrahedral elements (each of which with all four nodes on the surrogate boundary) is highly unlikely. Isolated abnormal boundary tetrahedrons are possible, but intermixed with many normal boundary tetrahedrons, and Assumption~\ref{assu:3} would then hold.
\end{rem}
\begin{rem}
Although we restrict the numerical analysis to a pure Dirichlet problem, we will present numerical examples with mixed Dirichlet/Neumann boundary conditions.
\end{rem}

\begin{lemma}
\label{lemma:inf_sup}
Assume Assumptions~\ref{assu:1}, ~\ref{assu:2}, and~\ref{assu:3} hold. Then there exists $\gamma > 0$ depending only on  $C_{\max}$, $\beta$, and the regularity of the mesh such that
\[ \inf_{u_h \in V^k_h(\tO)} \,  \sup_{v_h \in V^k_h(\tO)}  \, \frac{a^k_h(u_h \, , \, w_h)}{\|u_h \|_h \|v_h \|_h} \geq \gamma \; . \]
\end{lemma}

\begin{proof}
Take any $u_h \in V^1_h(\tO)$. We construct first a test function $w_h \in V^1_h(\tO) \subset V^k_h(\tO)$ such that
\begin{subequations}
\begin{equation}
    \langle \,  \tn \cdot \nabla w_h, \tS^k u_h \, \rangle_{\tG} \geq c_1 \, \| h^{-1} \, u_h \|_{0, \tG}^2 - c_2 |u_h |_{1, \tO}^2
    \label{wh1}
\end{equation}
and
\begin{equation}
    \|w_h \|_h \leq c_3 \|u_h \|_h \label{wh2}
    \; .
  \end{equation}
\end{subequations}
This is achieved by taking $w_h \in V^1_h(\tO)$ such that $w_h$ vanishes at all the interior nodes and
\[ w_h = \ssf{I}^1_h u_h \; , \qquad \text{on } \tG \; ,  \]
where $\ssf{I}^1_h$ the nodal linear interpolation operator to $V^1_h(\tO)$.
We will restrict the proof of \eqref{wh1} to the two-dimensional case, since the three-dimensional case would follow very similar steps, although with more tedious derivations. 
Observe first that 
\[ \langle \,  \tn \cdot \nabla w_h, \tS^k u_h \, \rangle_{\tG} = \langle \,  \tn \cdot \nabla w_h, u_h \, \rangle_{\tG} - \langle \,  \tn \cdot \nabla w_h, u_h - \tS^k u_h \, \rangle_{\tG} \; , \]
and consider now an edge $\ti{e} \subset \tG \cap \partial \ti{T}$, for a corresponding element $\ti{T}$.
Applying Assumption~\ref{assu:1}, appropriate inverse and trace inequalities (which also rely on the equivalency of discrete norms), we obtain
\[ \| \tS^k u_h - u_h \|_{0,\ti{e}} 
\lesssim h_{\ti{e}}^{3 / 2} \| \nabla u_h \|_{L^{\infty} (\ti{e})} 
\lesssim h_{\ti{e}}^{1 / 2} \| \nabla u_h \|_{0, \ti{e}} 
\leqslant Ch_{\ti{T}}^{1 / 2} |u_h |^{\nosymbol}_{1, \ti{T}} 
\; , \]
where the symbol $\lesssim$ indicates inequality up to a constant.
By analogous inverse inequalities, since $w_h$ is fully determined by $u_h$ on the boundary $\tG$, we have
\[ \| \tn \cdot \nabla w_h \|_{0,\ti{e}} \leqslant \frac{C}{h_{\ti{T}}} \| u_h \|_{0,\ti{e}} \; . \]
Hence
\begin{equation}
    \label{ExtraTermTau} \langle \,  \tn \cdot \nabla w_h, \tS^k u_h \, \rangle_{\tG} \geqslant \langle \,  \tn \cdot \nabla w_h, u_h \, \rangle_{\tG} - \frac{C}{\sqrt{h}} \| u_h \|_{0, \tG} | u_h |_{1, \tO}
\end{equation}
and to prove \eqref{wh1} it is sufficient to bound from below $\langle \,  \tn \cdot \nabla w_h, u_h \, \rangle_{\tG}$, which can be done edge by edge. 
First, consider any boundary edge ${\ti{e}}$ associated with a {\it normal} boundary triangle ${\ti{T}}$.
Let $x_1, x_2$ be the endpoints of ${\ti{e}}$ and $x_3$ the remaining vertex of ${\ti{T}}$.
Since ${\ti{T}}$ is a normal boundary triangle, we are sure that $x_3$ is not on the boundary, hence $w_h (x_3) = 0$.
Let $x_{\ti{T}}$ be the point on the segment $x_1 x_2$ such that the segments $x_{\ti{T}} x_3$ and $x_1 x_2$ are perpendicular and let $h_{\ti{T}}$ be the distance between $x_{\ti{T}}$ and $x_3$ (i.e., the height of ${\ti{T}}$).
  Recalling that $w_h (x_3) = 0$ we observe
  \[ \tn \cdot \nabla w_h = \frac{w_h (x_{\ti{T}})}{h_{\ti{T}}} = \frac{\ssf{I}^1_h u_h (x_{\ti{T}})}{h_{\ti{T}}} \]
  with $\ti{\bs{n}}$ the outward unit normal on the edge ${\ti{e}}$. We have thus
  \begin{equation}
    \int_{\ti{e}} (\tn \cdot \nabla w_h) \, u_h = \int_{\ti{e}} \frac{\ssf{I}^1_h u_h (x_{\ti{T}})}{h_{\ti{T}}} \, u_h =
    \int_{\ti{e}} \frac{u_h^2}{h_{\ti{T}}} - \int_{\ti{e}} \frac{u_h - \ssf{I}^1_h u_h (x_{\ti{T}})}{h_{\ti{T}}} \, u_h
    \; .
    \label{intw}
  \end{equation}
  As $x_{\ti{T}}$ lies on the straight line passing through $x_1$ and $x_2$, there holds   $x_{\ti{T}} = ax_1 + (1 - a) x_2$ with  $a\in\mathbb{R}$ satisfying $\max(|a|,|1-a|)\le M$ with some $M>0$ depending only on the mesh regularity (in fact, in the most typical case of an acute triangle ${\ti{T}}$, one actually has $a\in [0, 1]$, but we also allow for obtuse mesh cells). Since $\ssf{I}^1_hu_h$ is an affine function on $\ti{e}$ taking the value $u_h(x_1),u_h(x_2)$ at points $x_1,x_2$ respectively,  we have
  \[ \ssf{I}^1_h u_h (x_{\ti{T}}) = au_h (x_1) + (1 - a) u_h (x_2) \]
  and, for all $x \in {\ti{e}}$,
  \[ |u_h (x) - \ssf{I}^1_hu_h (x_{\ti{T}}) | \leq |a|  |u_h (x) - u_h (x_1) | + |1 - a| |u_h (x) - u_h (x_2) | \leq 2M h_{\ti{e}} \max_{x \in {\ti{T}}} | \nabla u_h (x) | \; . \] 
%
  Using the inverse inequality
\[ \max_{x \in {\ti{T}}} | \nabla u_h (x) | \leq \frac{C_I}{h_{\ti{e}}^{\nosymbol}} |u_h
|^{\nosymbol}_{1, {\ti{T}}} \; , \]
it follows that
\[ \|u_h - u_h (x_{\ti{T}}) \|_{0, {\ti{e}}}  \leq C \sqrt{h_{\ti{T}}} |u_h |^{\nosymbol}_{1, {\ti{T}}} \; , \]
which can be substituted inside \eqref{intw} to yield
\begin{align*}
	\int_{\ti{e}} (\tn \cdot \nabla w_h) \, u_h & \geq
	\frac{1}{h_{\ti{T}}} \|u_h \|^2_{0, {\ti{e}}} - \frac{C}{\sqrt{h_{\ti{T}}}} |u_h|_{1, {\ti{T}}} \|u_h\|_{0, {\ti{e}}} \; .
\end{align*}
Summing this over all the boundary edges belonging to normal boundary cells
and combining with \eqref{ExtraTermTau} gives
	\begin{equation}\label{Almostwh1} \langle \,  \tn \cdot \nabla w_h, \tS^k u_h \, \rangle_{\tG} \geqslant
	 \| h^{-1/2} u_h \|^2_{0, \tG^n} - {C} \, \| h^{-1/2} u_h \|_{0, \tG} | u_h |_{1, \tO} - \langle \,  \tn \cdot \nabla w_h, u_h
	\, \rangle_{\tG^a} 
	\end{equation}
	where $\tG^n$ regroups the boundary edges from normal boundary cells,
	and $\tG^a$ regroups the boundary edges from abnormal boundary cells. This would immediately lead to \eqref{wh1} in the case $\tG=\tG^n$, that is when all the boundary cells are normal. 
	
However, $\tG^a\not=\emptyset$ in general and given any abnormal boundary cell ${\ti{T}}$ with all three vertices on $\tG$, we have $w_h = \ssf{I}^1_h u_h$, on the whole of ${\ti{T}}$.
	Hence, $| w_h |_{1, {\ti{T}}} \leqslant C | u_h |_{1, {\ti{T}}}$ and consequently $\| \ti{\bs{n}}
	\cdot \nabla w_h \|_{0, \partial {\ti{T}}} \leqslant C \, h_{\ti{T}}^{-1/2} | u_h |_{1, {\ti{T}}}$. Applying this to all abnormal boundary cells gives
	\begin{equation} \label{Gammahan1}
		\langle \,  \tn \cdot \nabla w_h, u_h \, \rangle_{\tG^a} \leqslant
	{C} |  u_h |_{1, \tO} \| h^{-1/2} u_h \|_{0, \tG^a} 
	\; .
	\end{equation}
In view of Assumption~\ref{assu:3}, to pass from the norm on $\tG^a$ to that on $\tG^n$, let us consider a
boundary edge $\ti{e}^a$ from an abnormal cell sharing a vertex $v$ with an edge
$\ti{e}^n$ from a normal cell. We have then
\begin{align}
\|u_h \|_{0, \ti{e}^a}^2 \leqslant C (h_{\ti{e}^a} | u_h (v)  |^2 + h_{\ti{e}^a}^2 \| \nabla u_h
\|^2_{0, {\ti{e}^a}}) \leqslant C (\| u_h \|^2_{0, \ti{e}^n} + h_{{\ti{e}^n}}^2 \| \nabla u_h
\|^2_{0, {\ti{e}^n}} + h_{\ti{e}^a}^2 \| \nabla u_h \|^2_{0, {\ti{e}^a}}) 
\; . 
\label{eq:bau}
\end{align}
If a chain of connected abnormal boundary edges is attached to a normal edge, similar inequalities hold, with constants that depend on the length of such a chain. 
Note in particular that this length is uniformly bounded, under Assumptions~\ref{assu:1}, ~\ref{assu:2}, and ~\ref{assu:3}.
Summing~\eqref{eq:bau} on all the abnormal edges gives
	\begin{equation}
		\label{Gammahan} \|h^{-1/2}  u_h \|_{0, \tG^a} \leqslant C (\| h^{-1/2} u_h \|_{0,
			\tG^n} + \| h^{1/2}  \, \nabla u_h \|_{0, \tG}) \leqslant C \left( \| h^{-1/2} u_h
		\|_{0, \tG^n} + | u_h |_{1, \tO} \right)
		\; ,
	\end{equation}
which can be used in \eqref{Gammahan1} to obtain
	\[ \langle \,  \tn \cdot \nabla w_h, u_h \, \rangle_{\tG^a} \leqslant
	C \, | u_h |_{1, \tO} \| h^{-1/2}  u_h \|_{0, \tG^n} + C \, | u_h |_{1, \tO}^2 . \]
Equation \eqref{Almostwh1} now entails
	\[ \langle \,  \tn \cdot \nabla w_h, \tS^k u_h \, \rangle_{\tG} \geqslant
	 \| h^{-1/2}  u_h \|^2_{0, \tG^n} - {C}  \, \| h^{-1/2} u_h \|_{0, \tG} | u_h |_{1, \tO} - C  \, | u_h |_{1, \tO}^2 \; . \]
Moreover, \eqref{Gammahan} can be rewritten as
	\[ \| h^{-1/2} u_h \|_{0, \tG}^2 \leqslant C  \,\left(\| h^{-1/2} u_h \|^2_{0, \tG^n} + |u_h |^2_{1, \tO} \right) \; . \]
Hence,
	\[ \langle \,  \tn \cdot \nabla w_h, \tS^k u_h \, \rangle_{\tG}
	\geqslant \frac{1}{C}  \, \| h^{-1/2} u_h \|^2_{0, \tG} - {C}  \, \|h^{-1/2} u_h \|_{0, \tG} | u_h |_{1, \tO} - C  \, | u_h |_{1, \tO}^2 \; , \]
which, with Young inequality, yields \eqref{wh1}. 

\begin{rem}
Note that there are also boundary triangles, with only one node on $\tG$ and two nodes on the interior. These triangles pose no problem in this part of the proofs, since they do not contribute to boundary integrals. They will have role later on in the convergence proofs.
\end{rem}

Proving \eqref{wh2} is easy. Indeed,
\[ \|w_h \|_h = \left( |w_h |_{1, \tO}^2 + \| h^{-1/2} \, \ssf{I}^1_h u_h
\|_{0, \tG}^2 \right)^{\frac{1}{2}} \; . \]
By inverse inequalities (equivalence of norms on every facet),
\[ \|h^{-1/2} \, \ssf{I}^1_h u_h \|_{0, \tG} \leqslant C \|h^{-1/2} u_h \|_{0, \tG} \]
Moreover, recalling that $w_h$ is fully determined on any cell having a facet (normal or abnormal) on $\tG$ by its values on the $\tG$-part of this cell boundary, we get by equivalence of norms on such cells and on their adjacent cells,
\[ |w_h |_{1, \tO} \leqslant C \| h^{-1/2} w_h \|_{0, \tG}
= {C} \|h^{-1/2} \, \ssf{I}^1_h u_h \|_{0, \tG}  \; , \] 
Hence,
\[ \|w_h \|_h \leqslant {C} \, \| h^{-1/2} u_h \|_{0, \tG} \leqslant C  \, \|u_h \|_h \; .\]

Taking now $v_h = u_h + \alpha w_h$ yields
\[ a^k_h (u_h, v_h) = (\nabla u_h, \nabla u_h)_{\tO} + \langle \,  \tn \cdot \nabla u_h, \tS^k u_h - u_h \, \rangle_{\tG}
+ \alpha (\nabla u_h, \nabla w_h)_{\tO} - \alpha \langle \,  \tn \cdot \nabla u_h, w_h \, \rangle_{\tG} + \alpha \langle \,  \tn \cdot \nabla w_h, \tS^k u_h \, \rangle_{\tG} \; . \]
Observe also that
\[ | \langle \,  \tn \cdot \nabla u_h, \tS^k u_h - u_h \, \rangle_{\tG} |
\leqslant \|\tn \cdot \nabla u_h \|_{0, \tG}  \|\tS^k u_h - u_h
\|_{0, \tG} \leqslant \beta |u_h |_{1, \tO}^2 \]
thanks to Assumption~\ref{assu:2}.
We thus conclude by making use of all the previous estimates:

\begin{align*}
	a^k_h (u_h, v_h) 
	&\geq\;
	(1 - \beta) |u_h |_{1, \tO}^2 + \alpha c_1 \| h^{-1/2}  u_h \|_{0, \tG}^2 - \alpha c_2 |u_h |_{1,\tO}^2
	- \alpha |u_h |_{1, \tO}^{\nosymbol}  |w_h |_{1,\tO}^{\nosymbol} 
	\\
	&\phantom{\geq}\;
	- {\alpha C_I} |u_h |_{1,\tO}^{\nosymbol} \| h^{-1/2}  w_h \| \nobracket_{0, \tG}
	\\
	&\geq\;
	\left(1 - \beta - \alpha \frac{1 + C_I^2}{2 \varepsilon} - \alpha c_2 \right) |u_h |_{1, \tO}^2 - \frac{\alpha \varepsilon}{2} \|w_h \|_h^2 +
	\alpha c_1 \| h^{-1/2}  u_h \|_{0, \tG}^2
	\\
	&\geq\;
	\left( 1 - \beta - \alpha \frac{1 + C_I^2}{2 \varepsilon} - \alpha
	c_2 - \frac{\alpha \varepsilon c_3^2}{2} \right) |u_h |_{1, \tO}^2 +
	\left( \alpha c_1 - \frac{\alpha \varepsilon c_3^2}{2} \right) 
	\| h^{-1/2}  u_h \|_{0, \tG}^2 \\
	&\geq\;
	c_4 \, \|u_h \|_h^2 \; .
\end{align*}
The last bound is achieved by taking $\varepsilon = {c_1}/{c_3^2}$ and then
$\alpha$ sufficiently small, so that
$$c_4 \assign \min \left( 1 -
\beta - \alpha c_3^2  \frac{1 + C_I^2}{2 c_1} - \alpha c_2 - \frac{\alpha
	c_1}{2}, \alpha \frac{c_1}{2} \right)$$
is positive (note that $\beta <  1$). Recall also that
\[ \|v_h \|_h \leq (1 + \alpha c_3) \, \|u_h \|_h \; , \]
so that
\[ \frac{a^k_h (u_h, v_h)}{\|v_h \|_h} \geq \frac{c_4}{1 + \alpha c_3} \, \|u_h \|_h \; . \]
\end{proof}


We are now in a position to prove an \textit{a priori} error estimate, and for this purpose we require:
\begin{assumption}
	\label{assu:forInterpolation}
There is an $h$-independent constant $c > 0$, such that, for any measurable $\ti{\omega} \subset \tG$, one has that $\mathrm{meas}_{d-1} (\ti{\omega}) \leqslant c \, \mathrm{meas}_{d-1} (\bs{M}_{h} (\ti{\omega}))$ where $\mathrm{meas}_{d-1}$ denotes the $(d - 1)$-dimensional surface measure on $\tG$ or $\Gamma$.
\end{assumption}
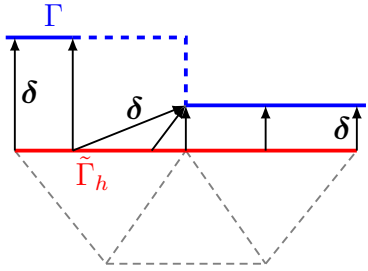
\begin{figure}[h!]
  \centering
    \begin{tikzpicture}[scale= 0.6]
      \draw[line width = 0.25mm,densely dashed,gray] (0,2.5) -- (2,0);
      \draw[line width = 0.25mm,densely dashed,gray] (2,0) -- (3.75,2.5);
      \draw[line width = 0.25mm,densely dashed,gray] (3.75,2.5) -- (5.5,0);
      \draw[line width = 0.25mm,densely dashed,gray] (2,0) -- (5.5,0);
      \draw[line width = 0.25mm,densely dashed,gray] (5.5,0) -- (7.5,2.5);
      
      \draw [line width = 0.5mm,blue, name path=true] plot[sharp corners] coordinates {(-0.2,5)(1.27,5)};
      \draw [line width = 0.5mm,blue, name path=true] plot[sharp corners] coordinates {(3.75,3.5)(7.7,3.5)};
      \draw [line width = 0.5mm,blue, dashed, name path=true] plot[sharp corners] coordinates {(1.27,5)(3.75,5)(3.75,3.5)};
      \draw[line width = 0.5mm,red] (0,2.5) -- (7.5,2.5);
      \node[text width=0.5cm] at (1.75,2.) {\large${\color{red}\tG}$};
      \node[text width=0.5cm] at (1.05,5.5) {\large${\color{blue}\G}$};
      \node[text width=0.5cm] at (0.55,3.8) {\large$\bs{\delta}$};
      \node[text width=0.5cm] at (2.85,3.45) {\large$\bs{\delta}$};
      \node[text width=0.5cm] at (7.4,3.0) {\large$\bs{\delta}$};
      \draw[-stealth,line width = 0.25mm,-latex] (3.75,2.5) -- (3.75,3.5);
      \draw[-stealth,line width = 0.25mm,-latex] (1.27,2.5) -- (3.75,3.5);
      \draw[-stealth,line width = 0.25mm,-latex] (3.,2.5) -- (3.75,3.5);
      \draw[-stealth,line width = 0.25mm,-latex] (0,2.5) -- (0,5);
      \draw[-stealth,line width = 0.25mm,-latex] (1.27,2.5) -- (1.27,5);
      \draw[-stealth,line width = 0.25mm,-latex] (7.5,2.5) -- (7.5,3.5);
      \draw[-stealth,line width = 0.25mm,-latex] (5.5,2.5) -- (5.5,3.5);
      
    \end{tikzpicture}
    \caption{An anomalous case avoided by Assumption~\ref{assu:forInterpolation}: the distance vector $\bs{\delta}(\ti{\bs{x}})$ has a jump as $\ti{\bs{x}}$ runs along $\bs{\G}$. The map $\bs{M}_h$ does not map the dashed portion of the boundary $\G$.}
    \label{fig:gammaJump}
\end{figure}
Essentially, the previous assumption requires that when a point $\ti{\bs{x}}$ runs over the surrogate boundary $\tG$ with unit ``speed,'' its image $\bs{x}=\bs{M}_{h}(\ti{\bs{x}})$ runs over $\Gamma$ with a speed bounded from below.
In practical terms, Assumption~\ref{assu:forInterpolation} wants to avoid situations like the ones depicted in Figure~\ref{fig:gammaJump}, in which part of the true boundary is not mapped by the surrogate boundary. 
In three dimensions, the previous assumption also implies that $\bs{M}_{h}$ cannot map perpendicular paths on $\tG$ onto almost parallel paths on $\Gamma$. 

Assumption~\ref{assu:forInterpolation} is however less stringent than the assumption made in~\cite{atallah2022high, atallah2021analysis}, which requires the local coordinate system induced by $\bs{M}_{h}$ on the whole strip $\Omega\setminus\tO$ to be well posed.
In contrast, Assumption \ref{assu:forInterpolation} only requires the non-singularity of $\bs{M}_{h}$ as the map from $\tG$ to $\Gamma$. 
In particular, the case of true domains with edges/corners/vertices is included in Assumption \ref{assu:forInterpolation}, since $\bs{M}_{h}$ will be bounded and continuous, possibly with discontinuous derivatives.

\begin{thm}\label{aprioriH1}
	Let $u_h \in V_h^k(\tO)$ be the solution of the discrete problem (\ref{eq:SB_Poisson_uns}) and $u \in H^{k + 1} (\Omega)$ be the solution of the infinite dimensional problem \eqref{eq:SteadyPoisson}.
	Under the assumptions above, one has the error estimate
	\begin{equation}
		\label{apriori} | u - u_h |_{1, \Omega_h} \leqslant C \left( \sum_{\ti{T} \in \ti{\cT}_h} h_{\ti{T}}^{2 k} | u |^2_{k + 1, {\ti{T}}} \right)^{1 / 2} \; ,
	\end{equation}
	with an $h$-independent constant $C > 0$.	
\end{thm}
\begin{proof}
	Recall that the Taylor expansion of order $k$ is exact on polynomials of
	degree $k$, so that the bilinear form
	(\ref{eq:UnsymmetricShiftedNitscheBilinearForm}) can be rewritten on $V_h^k(\tO)$
	as
	\begin{equation}
		\label{akhMod} a^k_h  (u_h \hspace{0.17em}, \hspace{0.17em} w_h) =
		(\hspace{0.17em} \nabla u_h \hspace{0.17em}, \hspace{0.17em} \nabla w_h
		\hspace{0.17em})_{\tO} - \avg{\nabla u_h \cdot \ti{\bs{n}}, \ w_h}_{\tG} + \avg{u_h \circ \bs{M}_{h},  \nabla w_h \cdot \ti{\bs{n}}}_{\tG} \; .
	\end{equation}
	The term $u_h \circ \bs{M}_{h}$ should be interpreted here on every
	boundary facet $\ti{e} \in \tG$ as
	\[ (u_h \circ \bs{M}_{h}) |_{\ti{e}} \nobracket = (\ssf{E}_{\ti{T}} (u_h) \circ
	\bs{M}_{h}) |_{\ti{e}} \nobracket  \; , \]
	where $\ti{T} \in \ti{\cT}_h$ is the mesh element having ${\ti{e}}$ as one of its	sides, and $\ssf{E}_{\ti{T}}$ denotes the extension of a polynomial on ${\ti{T}}$ by the same polynomial viewed as function on $\mathbb{R}^d$.

Observe now that multiplying the governing equation in problem~\eqref{eq:SteadyPoisson} by $w_h \in V_h^k(\tO)$ and integrating by parts leads to 
\begin{align}
\label{eq:Galerkin!}
a^k_h  (u  \hspace{0.17em}, \hspace{0.17em} w_h) = l_h (w_h) \; ,
\end{align}
where the bilinear form $a_h^k$ should be understood in the sense of~\eqref{akhMod}. 
\begin{rem}
There is an important point to be made here. The interpretation of the Taylor expansion of a polynomial function at $\ti{\bs{x}} \in \tG$ being equivalent to the evaluation of the same polynomial at $\bs{x} = \bs{M}_{h}(\ti{\bs{x}}) \in \G$ is extended in~\eqref{akhMod} to any (non-polynomial) function.
Under this premise, {\it the Shifted Boundary Method satisfies a Galerkin-orthogonality statement}~\eqref{eq:Galerkin!}, which is something that was not realized in the literature of the method up to this point.
\end{rem}

Subtracting~\eqref{eq:Galerkin!} from ~\eqref{akhMod} yields the familiar Galerkin orthogonality relation
	\[ a^k_h  (u_h - u \hspace{0.17em}, \hspace{0.17em} w_h) = 0 \; . \]
Adding and subtracting $\ssf{I}^k_h u$, the interpolate of $u$, yields 
	\[ a^k_h  (u_h - \ssf{I}^k_h u \hspace{0.17em}, \hspace{0.17em} w_h) = a^k_h(u  - \ssf{I}^k_h u \hspace{0.17em}, \hspace{0.17em} w_h) \; . \]
Using the {inf-sup} estimate of Lemma~\ref{lemma:inf_sup} and the definition of the bilinear form $a_h^k$ gives
	\begin{align}\label{infsupcor}
		\gamma \| u_h - \ssf{I}^k_h u \|_h 
		&\leqslant\; 
		\sup_{w_h \in V_h^k} 	\frac{a^k_h  (u  - \ssf{I}^k_h u \hspace{0.17em}, \hspace{0.17em} w_h)}{\|	w_h \|_h} 
		\nonumber \\
		&\leqslant\; 
		| u  - \ssf{I}^k_h u |_{1, \tilde{\Omega}_h} + \| h^{1/2} \nabla (u  - \ssf{I}^k_h u) \cdot \tilde{\bs{n}} \|_{0,\tG} 
		\nonumber \\
		&\phantom{\leqslant}\; 
		\quad + \sup_{w_h \in V_h^k} \frac{\left\langle (u-\ssf{I}^k_h u) \circ \bs{M}_{h} \hspace{0.17em}, \hspace{0.17em} \nabla w_h \cdot \tilde{\bs{n}} \right\rangle_{\tG}}{\| w_h	\|_h} 
		\; .
	\end{align}
The first two terms on the right-hand side of the inequality above are	bound by $ch^k | u |_{k + 1}$ thanks to the usual interpolation estimates.
The last term requires some more work and can be treated element by element, for any $w_h \in V_h^k(\tO)$, applying Cauchy's inequality and equivalence of discrete norms:
	\begin{align*}
		\frac{\left\langle (u  - \ssf{I}^k_h u) \circ \bs{M}_{h} \hspace{0.17em}, \hspace{0.17em} \nabla w_h \cdot \tilde{\bs{n}} \right\rangle_{\tG}}{\| w_h \|_h} 
		&\leqslant\; 
		\frac{\displaystyle \left( \sum_{\ti{e} \in \tG} {h^{- 1}_{\ti{e}}}  \| (u  - \ssf{I}^k_h u) \circ \bs{M}_{h} \|_{0, \ti{e}}^2 \right)^{1 / 2} 	\left( \sum_{\ti{e} \in \tG} {h _{\ti{e}}}  \|	\nabla w_h \cdot \tilde{\bs{n}} \|_{0, \ti{e}}^2 \right)^{1 / 2}}
		{\displaystyle \left(\sum_{\ti{T} \in \ti{\cT}_h} \| \nabla w_h \|_{0, {\ti{T}}}^2 \right)^{1 / 2}}
		\nonumber \\
		&\phantom{\leqslant}\; 
		\leqslant
		C \left( \sum_{\ti{e} \in \tG} h_{\ti{e}}^{-1}  \| (u  - \ssf{I}^k_h u) \circ \bs{M}_{h} \|_{0, \ti{e}}^2 \right)^{1 / 2} 
		\; ,
	\end{align*}
where the sum is taken over all the boundary facets $\ti{e} \in \tG$, each of which belongs to a mesh element ${\ti{T}}$, and we have used the trace inverse inequality $\| \nabla w_h \cdot \tilde{\bs{n}} \|_{0, \ti{e}} \leqslant \, C \, h_{\ti{e}}^{-1/2} \| \nabla w_h \|_{0, {\ti{T}}}^2$. 
Moreover, for every $\ti{e} \subset \tG \cap \partial {\ti{T}}$ with $\ti{T} \in \ti{\cT}_h$, let us denote by $S \supset {\ti{T}}$ the simplex obtained from ${\ti{T}}$ by a homothetic transformation with center $\ti{\ssb{C}} \in \ti{T}$ and scaling coefficient $\kappa > 1$ such that $\bs{M}_{h}(\ti{e}) \subset \G \cap S$ (see Fig.~\ref{fig:SimplexS}).
Observe that $\kappa$ is independent of $h$ and in general not too large, that is $\kappa=2$ is a possible choice in view of Assumption~\ref{assu:1}.

\begin{figure}[t]
  \centering
  \begin{tikzpicture}[scale=0.75]
    \draw[line width = 0.25mm,densely dashed,gray] (2,3.) -- (3.5,1.);
    \draw[line width = 0.25mm,densely dashed,gray] (3.5,1.) -- (5.,3.);

    \draw [line width = 0.5mm,violet, name path=true] plot[sharp corners] coordinates {(0,5)(7,5)(3.5,0)(0,5)};
    
    \draw [line width = 0.5mm,blue, name path=true] plot[smooth] coordinates {(0,3.5)(2,4)(5,4)(7.5,4.5)};
    \draw [line width = 1.0mm,teal, name path=true] plot[smooth] coordinates {(2,4)(5,4)};
    \draw [line width = 0.5mm,red, name path=true] plot[sharp corners] coordinates {(0,1.75)(2,3)(5,3)(7.5,3.5)};
    
    \node[text width=0.5cm] at (1.75,2.) {\large${\color{violet}S}$};
    \node[text width=0.5cm] at (3.2,4.4) {\large${\color{teal}\bs{M}_h(\ti{e})}$};
    \node[text width=0.5cm] at (3.75,2.0) {\large$\ti{T}$};
    \node[text width=0.5cm] at (3.7,2.7) {\large${\color{red} \ti{e}}$};
    \node[text width=0.5cm] at (2.5,3.5) {\large$\bs{\delta}$};
    \node[text width=0.5cm] at (6.5,2.75) {\large${\color{red}{\tG}}$};
    \draw[-stealth,line width = 0.25mm,-latex] (2,3) -- (2,4);
    \draw[-stealth,line width = 0.25mm,-latex] (3.,3) -- (3,4);
    \draw[-stealth,line width = 0.25mm,-latex] (4.,3) -- (4,4);
    \draw[-stealth,line width = 0.25mm,-latex] (5,3) -- (5,4);
    
  \end{tikzpicture}
  \caption{A boundary element $\tilde{T} \in  \ti{\cT}_h$ and the corresponding simplex $\bs{S}$.}
  \label{fig:SimplexS}
\end{figure}

Then, by Assumption~\ref{assu:forInterpolation} and by the trace inequality, we have
	\[ \| (u  - \ssf{I}^k_h u) \circ \bs{M}_{h} \|_{0, \ti{e}}^2 
	\lesssim
	\|u  - \ssf{I}^k_h u \|_{0, \bs{M}_{h} (\ti{e})}^2 
	\lesssim
	\| u  - {\ssf E}_{\ti{T}} (\ssf{I}^k_h u) \|_{0,\G \cap S}^2 
	\leqslant 
	C \| {\ssf E} u  - {\ssf E}_{\ti{T}} (\ssf{I}^k_h u) \|_{0, S} \| {\ssf E} u  - {\ssf E}_{\ti{T}} (\ssf{I}^k_h u) \|_{1, S} \; , \]
where ${\ssf E} u \in H^{k + 1} (\mathbb{R}^d)$ is the $H^{k + 1}$-extension of $u$ from $\Omega$ to the whole space $\mathbb{R}^d$ (cf.~\cite{Adams03}). 
Taking $l = 0$ or $1$, observe that
	\[ \| {\ssf E} u - {\ssf E}_{\ti{T}} (\ssf{I}^k_h u) \|_{l, S} 
	\leqslant 
	\| {\ssf E} u - \ssf{I}^k_S {\ssf E} u \|_{l, S} + \| \ssf{I}^k_S {\ssf E} u - {\ssf E}_{\ti{T}} (\ssf{I}^k_h u) \|_{l, S} \; , \]
where $\ssf{I}^k_S$ stands for the interpolation on the simplex $S$. By standard interpolation bounds, the properties of smooth extensions, and the equivalence of norms on ${\ti{T}}$ and $S$,
	\[ \|\ssf{I}^k_S {\ssf E} u - {\ssf E}_{\ti{T}} (\ssf{I}^k_h u) \|_{l, S} 
	\leqslant 
	C \| \ssf{I}^k_S {\ssf E} u - \ssf{I}^k_h u \|_{l, {\ti{T}}} 
	\leqslant 
	C (\| {\ssf E} u -\ssf{I}^k_S {\ssf E} u
	\|_{l, S} + \| u - \ssf{I}^k_h u \|_{l, {\ti{T}}}) 
	\leqslant 
	Ch_{\ti{e}}^{k + 1 - l} | {\ssf E} u |_{k + 1, S} 
	\; , \]
	hence
	\[ \| {\ssf E} u - {\ssf E}_{\ti{T}} (\ssf{I}^k_h u) \|_{l, S}
	\leqslant 
	C \, h_{\ti{e}}^{k + 1 - l} |
	{\ssf E} u |_{k + 1, S} \]
	and
	\[ \| (u  - \ssf{I}^k_h u) \circ \bs{M}_{h} \|_{0, \ti{e}}^2 
	\leqslant
	C \, h_{\ti{e}}^{2 k - 1} | {\ssf E} u |_{k + 1, S} \; .
	\]
	Summing over all the boundary facets gives
	\begin{align}
	\label{eq:last_sum}
	\sum_{\ti{e} \in\tG} h_{\ti{e}}  \| (u  - \ssf{I}^k_h u) \circ \bs{M}_{h} \|_{0, \ti{e}}^2 
	\leqslant 
	C \sum_{\ti{e} \in\tG} h_{\ti{e}}^{2 k} | {\ssf E} u |_{k+ 1, S} 
	\; .
	\end{align}
	Since the extension ${\ssf E} u$ can be constructed locally to respect the regularity bound $|{\ssf E} u|_{k+ 1, S}\leqslant C| u |_{k+1, S\cap\Omega}$, the sum in~\eqref{eq:last_sum} can be rewritten as the right hand side of (\ref{apriori}). 
	Combining this final result with (\ref{infsupcor}) concludes the proof.
\end{proof}

$L^2$-error estimates are derived next, using an Aubin-Nitsche duality argument.
The derivations have analogies with the proofs by Burman~\cite{burman12} in the context of body-fitted grids, and yield an identical theoretical estimate of the convergence rate, suboptimal by a half an order of accuracy.
This result, however, does not match numerical experience, since both the SBM and the penalty-free method described in~\cite{burman12} show optimal convergence in practical computations.

\begin{thm}\label{aprioriL2}
	Suppose that the assumptions of Theorem~\ref{aprioriH1} hold. 
	In addition, assume that the mesh $\mathcal{T}_h$ is quasi-uniform of meshsize $\mathfrak{h}$ (i.e. $h_T \sim \mathfrak{h}$, with $\mathfrak{h}$ fixed, for all $T \in \mathcal{T}_h$) and that $\partial \Omega$ is of class $\mathcal{C}^2$.
	Then, the following $L^2$-error bound holds:
	\[ \| u - u_h \|_{0, \tO} \leqslant C \, \mathfrak{h}^{k + 1/2} | u |_{k + 1,	\Omega} \; . \]
\end{thm}

\begin{proof}

	Introduce $w \in H^2 (\Omega)$, the solution to
	\begin{align*}
		- \Delta w & = \; (u - u_h) \mathbbm{1}_{\tO} \quad \; \text{ on } \Omega \; , \\
		w & = \;  0 \qquad \qquad \qquad \text{ on } \partial \Omega \; ,
	\end{align*}
	where $\mathbbm{1}_{\tO}$ is the indicator function of the set $\tO$.
	Then, we have $\| w \|_{2, \Omega} \leqslant C \| u - u_h \|_{0, \tO}$ and, integrating by parts,
	\begin{multline*} \| u - u_h \|_{0, \tO}^2 = (u - u_h, - \Delta w )_{\tO} =
		- \langle u - u_h, \nabla w  \cdot \ti{\bs{n}} \rangle_{\tG} + (\nabla (u -
		u_h), \nabla w)_{\tO} \\
		= a^k_h (u - u_h, w) - \langle u - u_h, \nabla w  \cdot \ti{\bs{n}}
		\rangle_{\tG} - \langle \tS^k (u - u_h), \nabla w  \cdot \ti{\bs{n}}
		\rangle_{\tG} + \langle \nabla (u - u_h) \cdot \ti{\bs{n}}, w 
		\rangle_{\tG} 
		\; .
	\end{multline*}
	Thanks to the Galerkin orthogonality property (\ref{eq:Galerkin!}),
	\begin{align} \label{5terms}
		\| u - u_h \|_{0, \tO}^2 
		&=\;
		 a^k_h (u - u_h, w - \ssf{I}^1_h w) 
		- \langle u	- u_h + \langle \tS^k (u -	u_h), \nabla w  \cdot \ti{\bs{n}} \rangle_{\tG} 
		+ \langle \nabla (u - u_h)
		\cdot \ti{\bs{n}}, w  \rangle_{\tG} 
		\nonumber \\
		& \leqslant \; 
		\underbrace{| u - u_h |_{1, \tO} | w - \ssf{I}^1_h w |_{1, \tO}}_{I} 
		+
		\underbrace{\left\| \frac{1}{\sqrt{h}} \tS^k (u - u_h) \right\|_{0, \tG}
			\left\| \sqrt{h} \nabla (w - \ssf{I}^1_h w) \cdot \ti{\bs{n}} \right\|_{0, \tG}}_{II}
		\nonumber \\
		& \phantom{\leqslant} \; 
		+
		\underbrace{\left\| \sqrt{h} \nabla (u - u_h) \cdot \ti{\bs{n}} \right\|_{0, \tG} \left\|
			\frac{1}{\sqrt{h}} (w - \ssf{I}^1_h w) \right\|_{0, \tG}}_{III} 
		\nonumber \\
		& \phantom{\leqslant} \; 
		+
		\underbrace{(\| u - u_h \|_{0, \tG} + \| \tS^k (u - u_h) \|_{0,\tG}) \| \nabla w \cdot \ti{\bs{n}} \|_{0, \tG}}_{IV}
		+
		\underbrace{\| \nabla (u - u_h) \cdot \ti{\bs{n}} \|_{0, \tG} \| w \|_{0, \tG}}_{V} 
		\; .
	\end{align}
	We recall that the proof of Theorem~\ref{aprioriH1} gives $\| u - u_h \|_h
	\leqslant C \, \mathfrak{h}^k  \, | \, u |_{k + 1, \Omega}$. This enables us to obtain optimal bounds for the first three terms above:
	\begin{description}
		\item[Term I:] Recalling that $| u - u_h |_{1, \tO} \leqslant \| u - u_h \|_h$ and using standard interpolation estimates, we have
		\[ | u - u_h |_{1, \tO} | w - \ssf{I}^1_h w |_{1, \tO} \leqslant C \, \mathfrak{h}^k \, | u |_{k + 1, \Omega} \,  | w |_{2, \tO} \; . \]
		
		\item[Term II:] By definition of $\tS^k$,
		\begin{align*}
			\left\| \frac{1}{\sqrt{h}} \tS^k (u - u_h) \right\|_{0, \tG}
			& \leqslant \;
			\left\| \frac{1}{\sqrt{h}} (u - u_h) \right\|_{0, \tG} +
			\left\| \frac{1}{\sqrt{h}} \sum_{i = 1}^k \frac{\ssf{D}^i_{\bs{\delta}}
				\hspace{0.17em} (u - u_h)}{i!} \right\|_{0, \tG} 
			\nonumber \\
			& \leqslant \;
			\| u - u_h \|_h + C \, \| h^{1/2} \, \nabla (u - u_h) \|_{0, \tG}
			\; .
		\end{align*}	
		The last bound is obtained using Assumption~\ref{assu:1} and combining Theorem~\ref{aprioriH1} with discrete trace inequalities, inverse inequalities, and interpolation estimates.  
		Indeed, for any $i$ between 1 and $k$,
		\begin{align*}	 
			\left\| \frac{1}{\sqrt{h}} \frac{\ssf{D}^i_{\bs{\delta}} \hspace{0.17em}
				(u - u_h)}{i!} \right\|_{0, \tG} 
			& \leqslant \;
			C \, \| h^{i - 1/2} \, \nabla^i (u - u_h) \|_{0, \tG} 
			\nonumber \\ 
			& \leqslant \;  
			C \,
			\left(
			\| h^{i - 1/2} \, \nabla^i (u - \ssf{I}^i_h u) \|_{0, \tG} + 
			\| h^{i - 1/2} \, \nabla^i (\ssf{I}^i_h u - u_h) \|_{0, \tG}
			\right) 
			\nonumber \\ 
			& \leqslant \;  
			C \, 
			\left( \mathfrak{h}^k \, | u |_{k + 1, \Omega} 
			+ \| h^{-1/2} \, (\ssf{I}^i_h u - u_h) \|_{0, \tG}
			\right) 
			\nonumber \\ 
			& \leqslant \;  
			C \, \mathfrak{h}^k \, | u |_{k + 1, \Omega} 
			\; ,
		\end{align*}	
		where $\nabla^i$ is the $i$th-order gradient.  
		Hence,
		\[ \left\| \frac{1}{\sqrt{h}} \; \tS^k (u - u_h) \right\|_{0, \tG}
		\leqslant 
		C \left( \mathfrak{h}^k \, | u |_{k + 1, \Omega} 
		+ \| h^{-1/2} (u - u_h) \|_{0, \tG} \right) 
		\leqslant 
		C \, \mathfrak{h}^k \, | u |_{k + 1, \Omega}
		\]
		and
		\[ \left\| \frac{1}{\sqrt{h}} \; \tS^k (u - u_h) \right\|_{0, \tG}
		\left\| \sqrt{h} \; \nabla (w - \ssf{I}^1_h w) \cdot \ti{\bs{n}} \right\|_{0, \tG}
		\leqslant C \;  \mathfrak{h}^{k+1} \; | u |_{k + 1, \Omega} \; | w |_{2, \tO} \; . \]
		
		\item[Term III:] Because this term is similar to Term II, it can be bounded in the same	manner.
	\end{description}

\noindent 	The remaining two terms in (\ref{5terms}) are responsible for the loss of $\sqrt{h}$ in our error estimate.
	
	\begin{description}
		\item[Term IV:]  Again, by inverse estimates (similar to Term II),
		\begin{align*}
		\| u - u_h \|_{0, \tG} + \| \tS^k (u - u_h) \|_{0, \tG}
		& \leqslant \;  
		2 \| u - u_h \|_{0, \tG} + \left\| \sum_{i = 1}^k \frac{\ssf{D}^i_{\bs{\delta}} \hspace{0.17em} (u - u_h)}{i!} \right\|_{0,\tG} 
		\nonumber \\ 
		& \leqslant \;  
		2 \, \mathfrak{h}^{1/2} \, \| u - u_h \|_h + C \, \mathfrak{h}^{k+1/2} \,  | u |_{k + 1, \Omega}  )
		\; .
		\end{align*}
		This gives already the error of order $\mathfrak{h}^{k + \frac 12}$. 
		Unfortunately, one cannot	gain another $\mathfrak{h}^{1/2}$, since the norm $\| \nabla w \cdot \ti{\bs{n}} \|_{0, \tG}$ can only be bounded by $\| \nabla w \|_{1, \tO}$, using the trace theorem.
		Thus,
		\[ (\| u - u_h \|_{0, \tG} + \| \tS^k (u - u_h) \|_{0,\tG}) \, \| \nabla w \cdot \ti{\bs{n}} \|_{0, \tG} 
		\leqslant C \, \mathfrak{h}^{k+1/2} \, | u |_{k + 1, \Omega} \, \| w \|_{2, \tO} \; . \]
		\item[Term V:] By calculations similar to those above, we have
		\[ \| \nabla (u - u_h) \cdot \ti{\bs{n}} \|_{0, \tG} 
		\leqslant C \, \mathfrak{h}^{k-1/2} \, | u |_{k + 1, \Omega} \; . \]
		To bound $\| w \|_{0, \tG}$, we introduce the distance function $\phi$ on
		$\Omega$, i.e. $\phi (x) = \tmop{dist} (x, \Gamma)$ for any $x \in \Omega$ and
		recall the Hardy-type inequality from \cite{duprez2020phi} (recall that $w$ vanishes on
		$\Gamma$)
		\[ \left\| \frac{w}{\phi} \right\|_{1, \Omega} \leqslant C \,  \| w \|_{2, \Omega} \; .
		\]
		By Assumption~\ref{assu:1} and the trace theorem, we obtain
		\[ 
		\| w \|_{0, \tG} 
		\leqslant \max_{\ti{\bs{x}} \in \tG} | \phi (x) |  \left\| \frac{w}{\phi} \right\|_{0, \tG} 
		\leqslant C \, \mathfrak{h}  \, \left\| \frac{w}{\phi} \right\|_{1, \tO} 
		\leqslant C \, \mathfrak{h} \, \| w \|_{2, \Omega} 
		\; .
		\]
		Hence
		\[ \| \nabla (u - u_h) \cdot \ti{\bs{n}} \|_{0, \tG} \| w \|_{0, \tG} 
		\leqslant 
		C \, \mathfrak{h}^{k+1/2} \, | u |_{k + 1, \Omega} \, \| w \|_{2, \tO} \; . \]
	\end{description}
	Assembling the bounds for all the terms in (\ref{5terms}) yields
	\[ \| u - u_h \|_{0, \tO}^2 
	\leqslant C \; \mathfrak{h}^{k+1/2} \,  | u |_{k + 1,\Omega} \, \| w \|_{2, \tO} 
	\leqslant C \; \mathfrak{h}^{k+1/2} \,  | u |_{k + 1,\Omega} \, \| u - u_h \|_{0, \tO} \; , \]
	which concludes the proof.
\end{proof}

\subsection{On condition numbers}
\label{sec:cond_numbers}
it is easy to see that the condition number of the matrix associated to the bilinear form $a_h^k$ in~\eqref{eq:UnsymmetricShiftedNitscheBilinearForm} scales like $\mathfrak{h}^{-2}$ on a quasi-uniform mesh of size $\mathfrak{h}$, similar to the case of a standard FEM on a fitted mesh. 
More precisely, denoting by $\mathbf{A}$ the matrix representing $a_h^k$ in the standard basis of $V_h^{k}(\tO)$, the condition number $\kappa (\mathbf{A}) \assign \| \mathbf{A} \|_2 \| \mathbf{A}^{- 1} \|_2$ satisfies
\begin{equation}\label{TheorCondNum}
 \kappa (\mathbf{A}) \lesssim \mathfrak{h}^{- 2} \; . 
\end{equation}
Here, $\| \cdot \|_2$ stands for the matrix (operator) norm associated to the vector $2$-norm (or, Euclidean norm) $| \cdot |_2$, namely:
\begin{align*}
	\| \mathbf{A} \|_2 
	= \sup_{\mathbf{v} \in \mathbb{R}^N} \sup_{\mathbf{w} \in \mathbb{R}^N} \dfrac{\mathbf{Av} \cdot \mathbf{w}}{|	\mathbf{v} |_2 | \mathbf{w} |_2}  \; .
\end{align*}

We give here a sketch of the proof. Associating any $v_h\in V_h^{k}(\tO)$ to a vector $\mathbf{v} \in \mathbb{R}^N$ (of $N$ degrees of freedom), we observe $\|v_h \|_{0, \tO} \sim
\mathfrak{h}^{d / 2} | \mathbf{v} |_2$ by the equivalence of norms on the quasi-uniform  mesh. The continuity of the form $a_h^k$ in the norm $\|\cdot\|_h$ leads to
\begin{align*}
	\| \mathbf{A} \|_2 
	&=\; \sup_{\mathbf{v} \in \mathbb{R}^N} \sup_{\mathbf{w} \in \mathbb{R}^N} \dfrac{\mathbf{Av} \cdot \mathbf{w}}{|	\mathbf{v} |_2 | \mathbf{w} |_2}  
	\\
	&\lesssim \, \mathfrak{h}^d \sup_{v_h\in V_h^{k}(\tO)} \; \sup_{w_h\in V_h^{k}(\tO)} \;  \dfrac{a_h^k (v_h, w_h)}{\| v_h \|_{0, \Omega_h} \| w_h \|_{0, \Omega_h}} 
	\\
	&\lesssim \, \mathfrak{h}^d \sup_{v_h\in V_h^{k}(\tO)} \; \sup_{w_h\in V_h^{k}(\tO)} \;  \dfrac{\| v_h \|_h \| w_h \|_h}{\| v_h \|_{0, \Omega_h} \| w_h \|_{0, \Omega_h}} 
	\\
	&\lesssim \, \mathfrak{h}^{d - 2} 
\end{align*}
	since $\| w_h \|_h \lesssim \mathfrak{h}^{-1} \, \| w_h \|_{0, \Omega_h} $.
	Similarly, the inf-sup condition of Lemma \ref{lemma:inf_sup} implies
\begin{align*}
   \frac{1}{\| \mathbf{A}^{- 1} \|_2} 
	= \inf_{\mathbf{v} \in \mathbb{R}^N} \sup_{\mathbf{w} \in \mathbb{R}^N} \dfrac{\mathbf{Av} \cdot \mathbf{w}}{| \mathbf{v} |_2 | \mathbf{w} |_2} 
	\gtrsim \, \mathfrak{h}^d \sup_{v_h\in V_h^{k}(\tO)} \; \sup_{w_h\in V_h^{k}(\tO)} \;  \dfrac{a_h^k (v_h, w_h)}{\| v_h \|_{0, \Omega_h} \| w_h \|_{0, \Omega_h}}  
	\gtrsim \, \mathfrak{h}^d 
\end{align*}
since $\| w_h \|_{0, \Omega_h} \lesssim \| w_h \|_h$ (a form of the Poincar\'{e} inequality). These two estimates give (\ref{TheorCondNum}).

\section{A penalty-free Shifted Boundary Method for compressible linear elasticity \label{sec:sbm_elasticity}}

In the numerical tests that follow, we also consider the equations of (compressible) linear elasticity. Their strong form is given as
\begin{subequations}
	\label{eq:SteadyLinEla}
\begin{align}
- \nabla \cdot \left( \bs{\sigma(\bs{u})} \right) &=\; \bs{b}  \qquad \text{\ \ in \ } \Om \; ,
 \\
\bs{u} &=\;  \bs{u}_D \qquad \! \text{on \ }  \G_D \; ,
 \\
\bs{\sigma} \bs{n} &=\;  \bs{t}_N \qquad \text{on \ }  \G_N \; ,
\end{align}
\end{subequations}
where $\bs{u}$ is the displacement field, $\bs{u}_D$ its value on the Dirichlet boundary $\G_D$, $\bs{t}_N$ the normal traction along the Neumann boundary $\G_N$, and $\bs{b}$ a body force. 
We assume that $\partial \Om = \overline{\G_D \cup \G_N}$ and $\G_D \cap \G_N = \emptyset$.
The stress $\bs{\sigma}$ is a linear function of $\bs{u}$, according to the constitutive model 
\[ \bs{\sigma}(\bs{u}) = \ssb{C} \, \bs{\varepsilon}(\bs{u}) \; , \]
where $\ssb{C}$ is the fourth-order elasticity tensor. For isotropic materials, the previous definition of the stress reduces to
$$\bs{\sigma}(\bs{u}) = 2 \mu \, \bs{\varepsilon}(\bs{u}) + \lambda (\nabla \cdot \bs{u}) \bs{I}\; . $$
We then consider the following variational formulation of~\eqref{eq:SteadyLinEla}, also inspired by~\cite{atallah2021shifted}: 
\begin{quote}
Find $\bs{u}_h \in \bs{V}_h(\tO)$, where 
$$\bs{V}_h(\tO)  : = \; \left\{ \bs{v}_h \in (C^0(\tO))^{n_{d}}  \ | \ {\bs{v}_h}_{|T} \in (\mathcal{P}^1(T))^{n_{d}}  \, , \, \forall T \in \ti{\mathcal{T}}_h \right\} \; , $$
such that,  for any $\bs{w}_h \in \bs{V}_h(\tO)$, it holds
\end{quote}
\vspace{-0.3cm}
%
\begin{align}
\label{eq:DiscreteShiftedNitscheVariationalFormSteadyLinElaIso}
( 2 \mu \, \bs{\varepsilon}(\bs{u}_h) \, , \, \bs{\varepsilon}(\bs{w}_h) )_{\tO} 
+ ( \lambda \, \nabla \cdot \bs{u}_h \, , \, \nabla \cdot \bs{w}_h )_{\tO} 
- \avg{ \, 2 \mu \, \bs{\varepsilon}(\bs{u}_h) \ti{\bs{n}}  + \lambda (\nabla \cdot \bs{u}_h) \ti{\bs{n}}  \, , \,  \bs{w}_h }_{\tG}
&
\nonumber \\
+ \avg{ \tS^k \bs{u}_h  - \bs{u}_D   \, , \,   2 \mu \, \bs{\varepsilon}(\bs{w}_h) \ti{\bs{n}}  + \lambda (\nabla \cdot \bs{w}_h) \ti{\bs{n}} }_{\tGD} 
&
\nonumber \\
+ \avg{ \bs{w}_h  , \, (\bs{n} \cdot \ti{\bs{n}}) \, ( 2 \mu \, \tS^{k-1} ( \, \bs{\varepsilon}(\bs{u}_h)) \bs{n} + \lambda \, \tS^{k-1} (\nabla \cdot \bs{u}_h) \bs{n}) - \bar{\bs{t}}_N }_{\tGN} 
& \; = \; 
(\bs{b} \, , \, \bs{w}_h )_{\tO} 
\, .
\end{align}

This variational formulation can be analyzed with very similar strategies to the Poisson equation, and similar results on convergence and stability can be derived in the case of pure Dirichlet boundary conditions.
In the case when Neumann boundary conditions are applied, numerical tests indicate that this formulation is sub-optimal by one order of accuracy, because the stress can only be extrapolated with Taylor expansions of order $k-1$, since it contains the gradients of $\bs{u}$.
We recall that a high-order body-fitted finite element method requires the fitting of the mesh to the geometry at the same order of accuracy of the polynomial interpolation space utilized.
This represents a strong practical limitation of high-order discretizations, and the tradeoff offered by the SBM with Neumann boundary conditions is that no body-fitted meshing is required, at the price of the loss of one order of accuracy with respect to optimal rates. 

Obviously, if only Dirichlet boundary conditions are applied, the SBM is optimal, although this situation is less common in structural mechanics applications, and more typical in fluid mechanics applications.

\begin{figure}[t!]
  \centering
  \begin{subfigure}[tbh!]{.3\textwidth}
    \centering
    \begin{tikzpicture}
      \filldraw[color=black, fill=black!10, thick](1,1) rectangle (3,3);
      \draw[color=black, thin](0.0,0.0) rectangle (4,4);
      \draw (0,0) -- (2,0);
      \node[text width=0.1cm] at (2,1-0.2){$l$};
    \end{tikzpicture}
    \caption{The domain $\Om$ (grey).}
    \label{fig:square}
  \end{subfigure}
  \qquad
  \begin{subfigure}[tbh!]{0.30\textwidth}\centering
    \includegraphics[width=\linewidth]{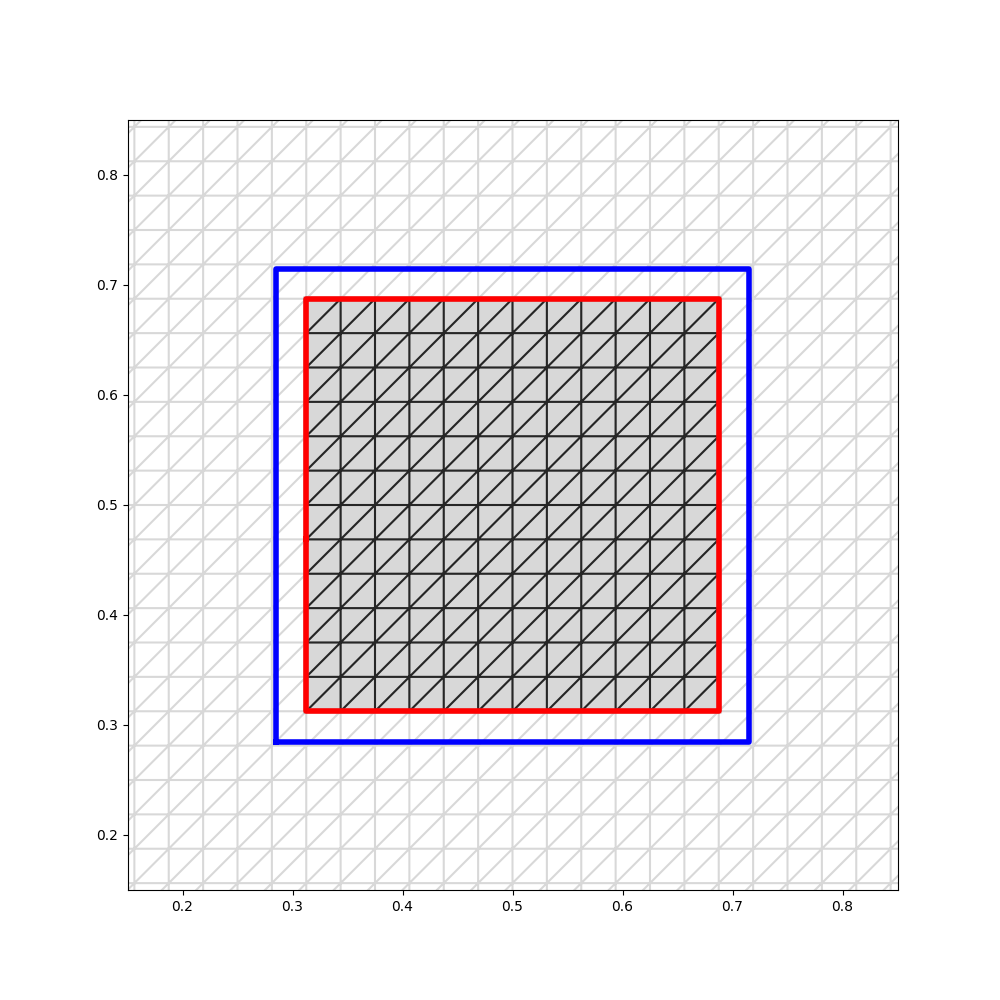}
    \caption{Domain and grid (no rotation).}
    \label{fig:mesh}
  \end{subfigure}
  \begin{subfigure}[tbh!]{0.30\textwidth}\centering
    \includegraphics[width=\linewidth]{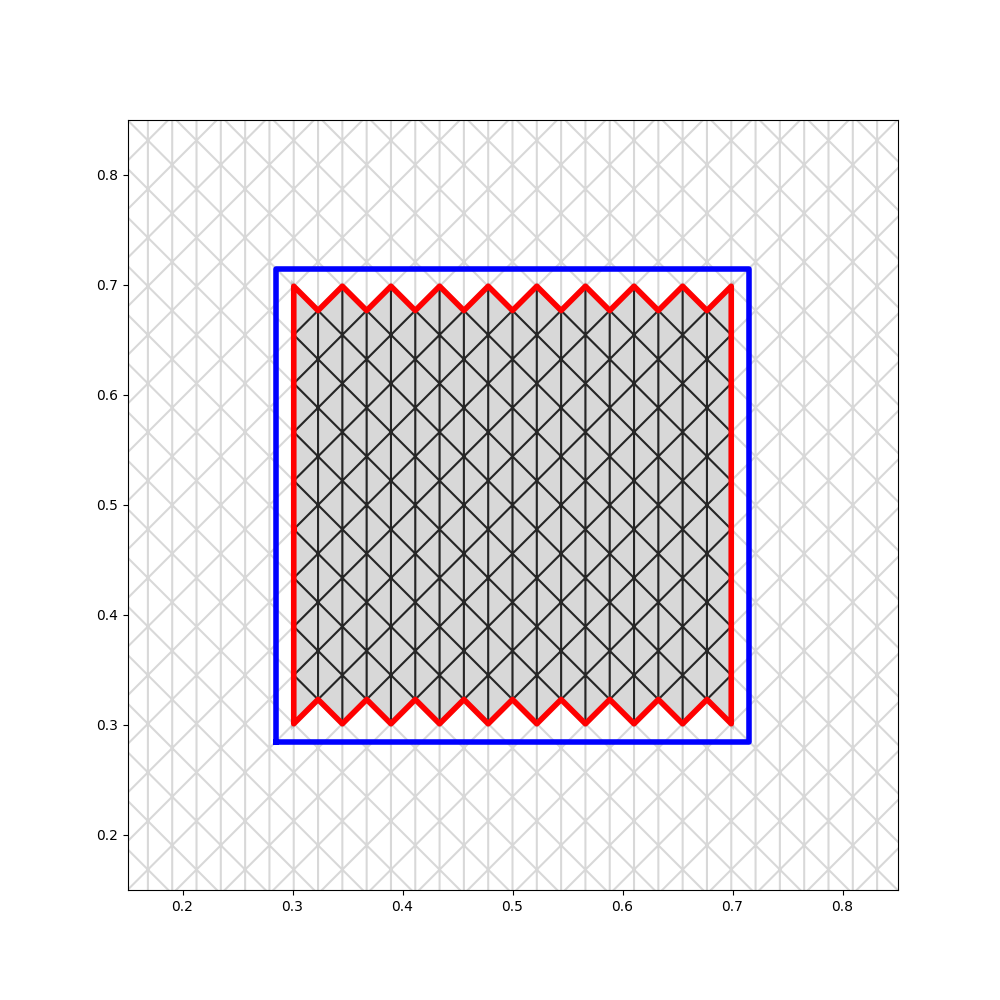}
    \caption{Domain and grid ($45^\circ$-rotation).}
    \label{fig:rotmesh}
  \end{subfigure}
\caption{Problem schematic for the square domain (grey), true boundary (blue), surrogate boundary (red) for the base grid (no rotation) and the the same grid rotated by 45 degrees.}
\end{figure}
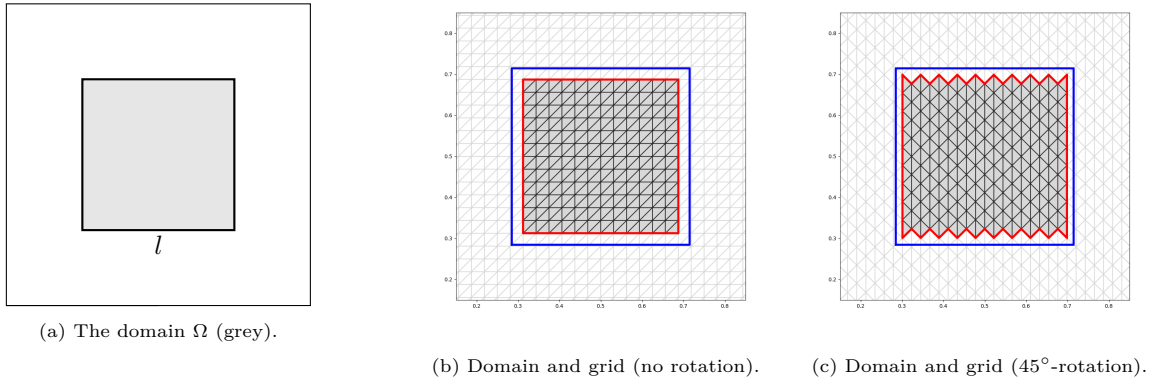
\begin{figure}[!htb]
	\centering
	\subfloat[$u_h$ for $0^o$-rotation]{\includegraphics[trim=2.5cm 2.5cm 2.5cm 2.5cm, clip, width = 1.5in]{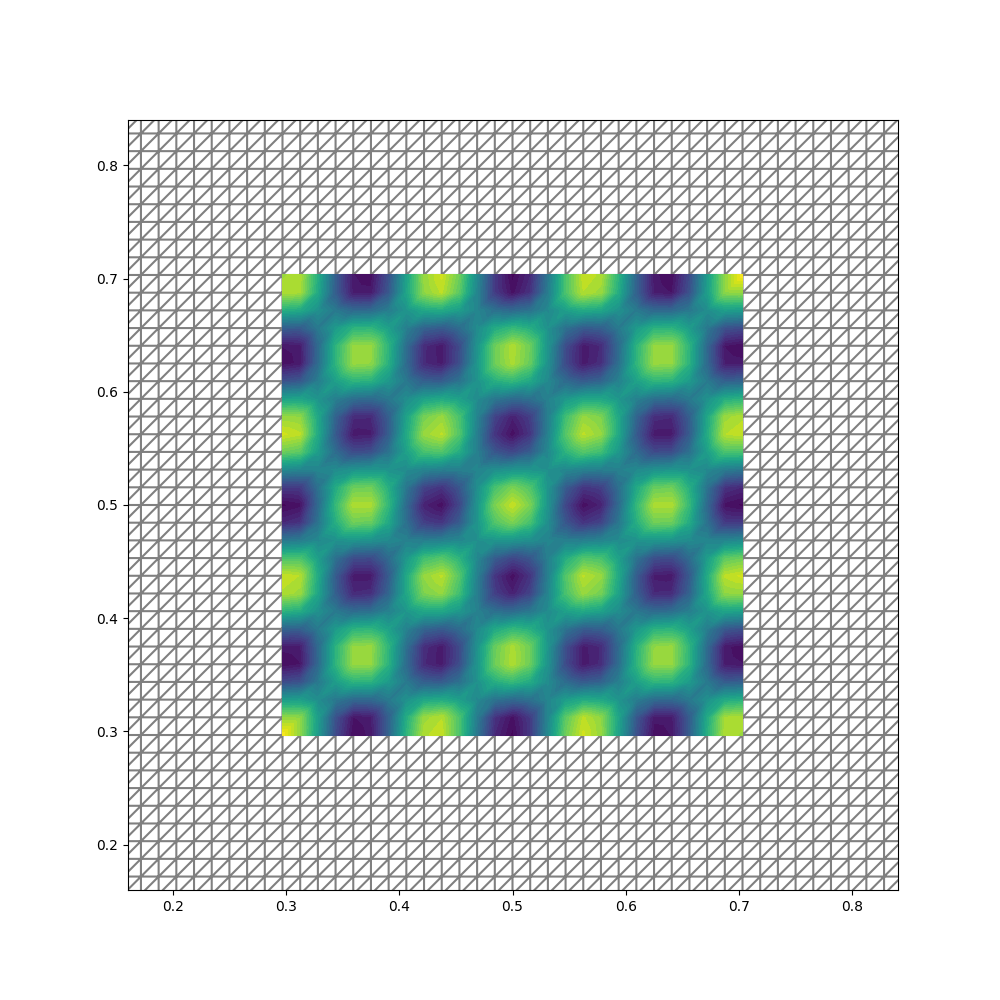}}
	 \; \;
	\subfloat[$u_h$ for $13.5^o$-rotation]{\includegraphics[trim=2.5cm 2.5cm 2.5cm 2.5cm, clip, width = 1.5in]{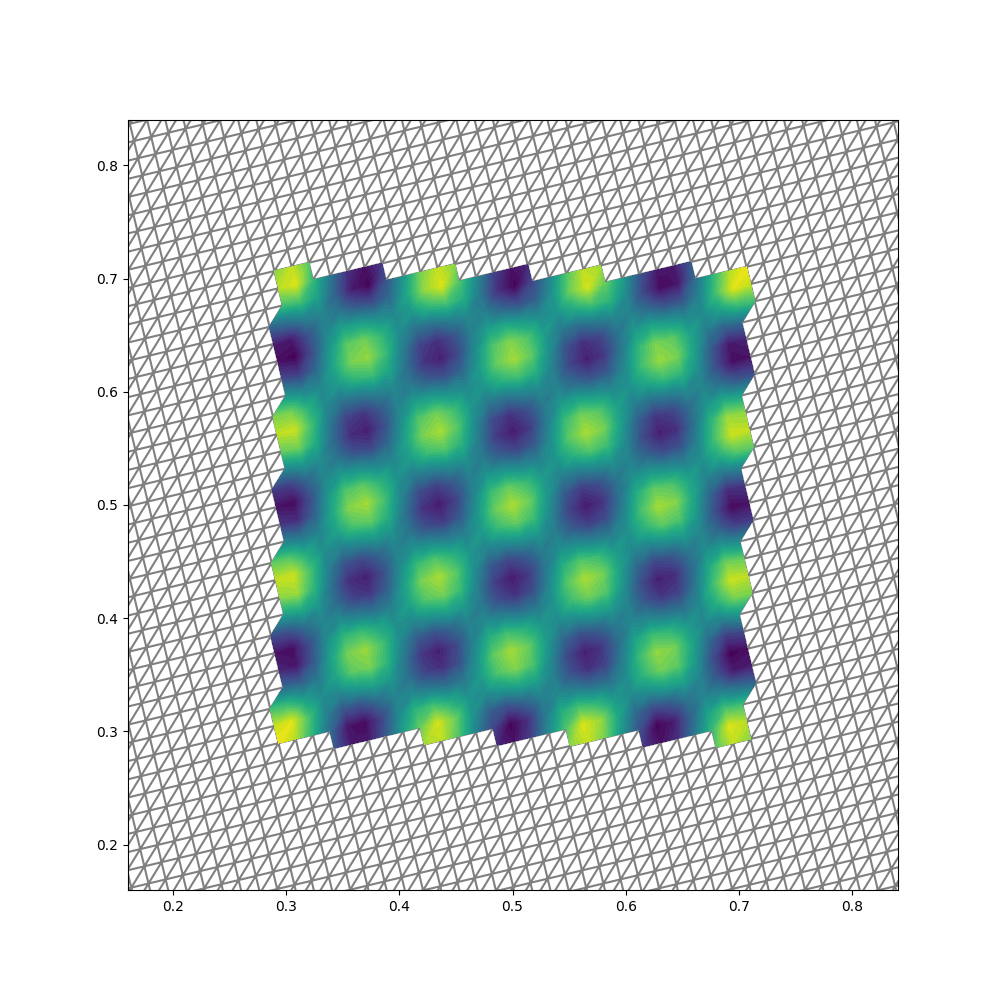}}
	 \; \;
	\subfloat[$u_h$ for $31.5^o$-rotation]{\includegraphics[trim=2.5cm 2.5cm 2.5cm 2.5cm, clip, width = 1.5in]{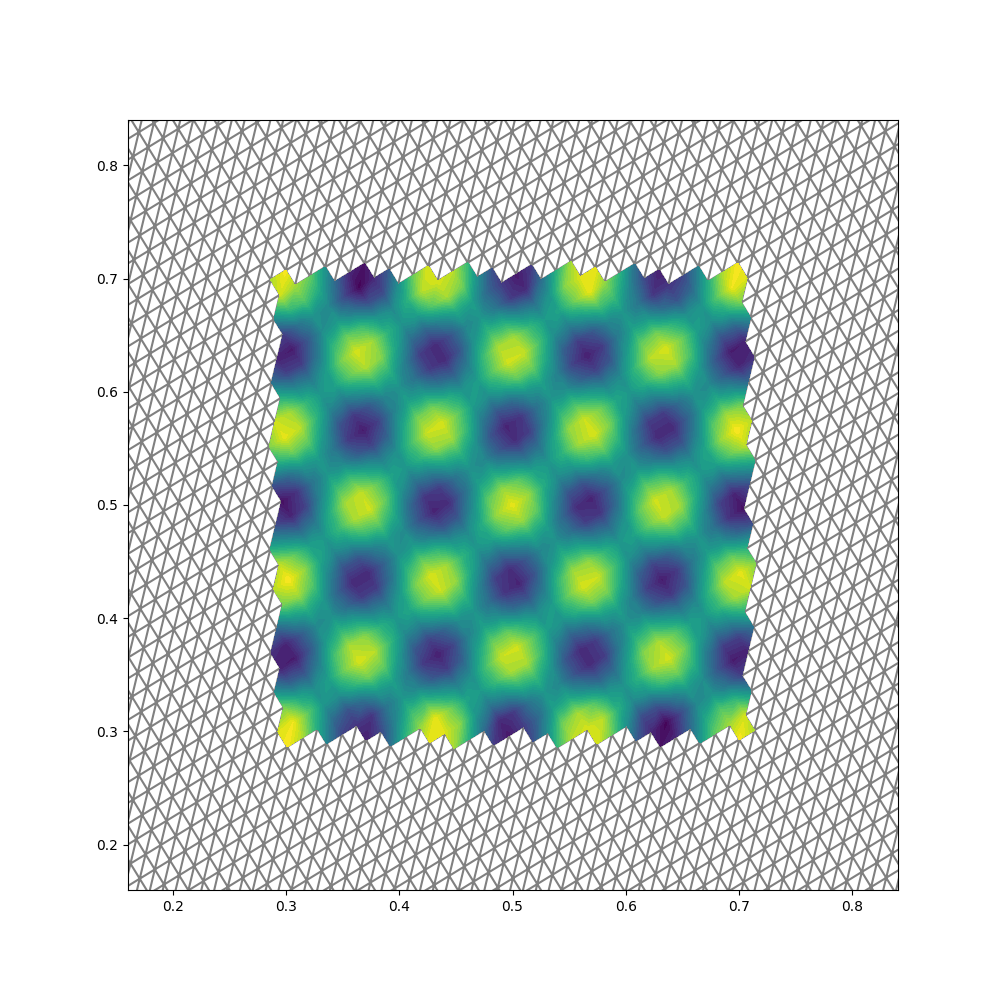}}
	 \; \;
	\subfloat[$u_h$ for $45^o$-rotation]{\includegraphics[trim=2.5cm 2.5cm 2.5cm 2.5cm, clip, width = 1.5in]{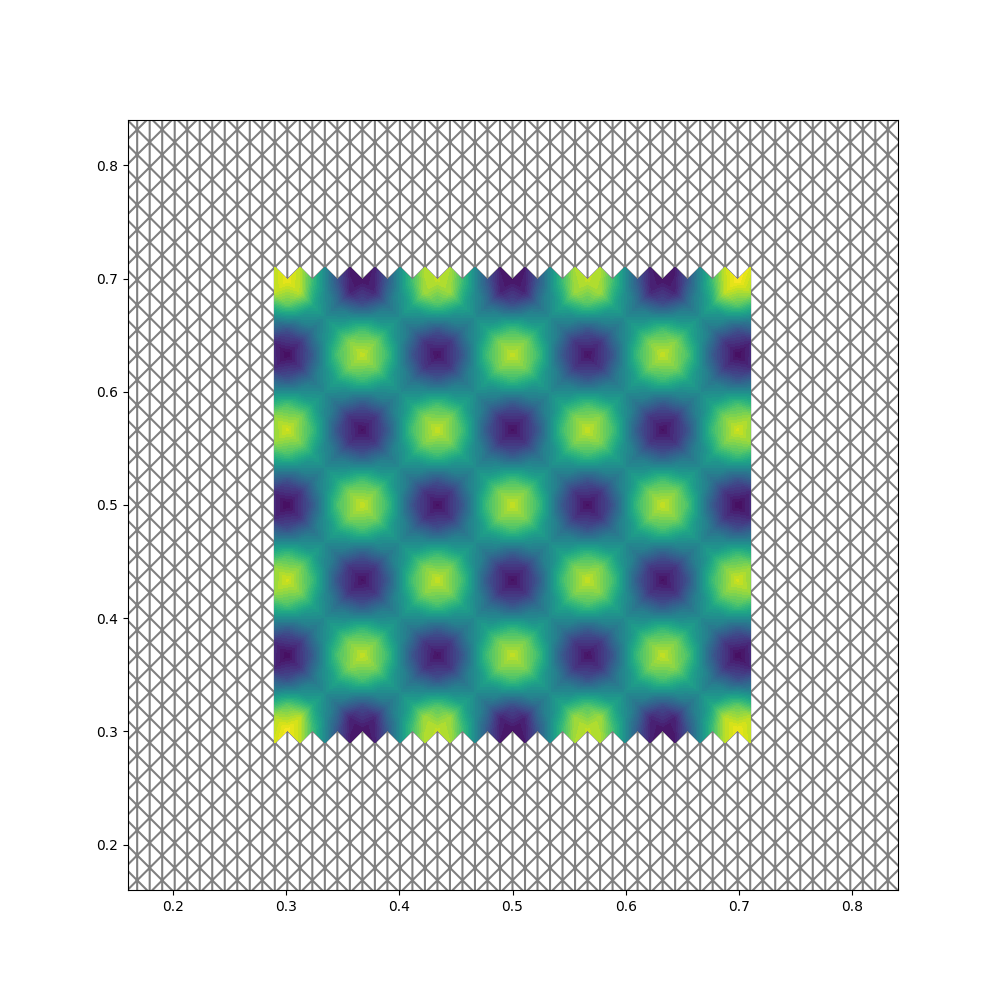}}
        \caption{Visualizations of the numerical solution $u_h$ for various rotations of the computational grid.}
        \label{fig:visual}

\end{figure}
\section{Numerical results \label{sec:numerical_results}}

\subsection{Poisson problem with Dirichlet boundary conditions}
A series of numerical tests were performed for the Poisson equation defined on a square domain $\Om$ of length $l = 0.43$ centered at $[0.5, 0.5]$. The problem schematic is depicted in \ref{fig:square}, with the domain in grey, the true boundary in blue, and the surrogate boundary in red. 
The analytical solution is
$$
  u(x,y) = \sin(15 \pi x) \sin(15 \pi y) \;
$$
and is obtained with the method of manufactured solutions, from the forcing term
$$
  f(x,y) = 450 \, \pi^2\sin(15 \pi x) \, \sin(15 \pi y) \; .
$$
Dirichlet boundary conditions that are compatible with the analytic solution were enforced on all four sides of the square.

In order to examine a high number of surrogate boundary arrangements, the square domain was immersed into a regular triangular background mesh, which was then incrementally rotated from 0 degrees (see Fig.~\ref{fig:mesh}) to 45 degrees (see Fig.~\ref{fig:rotmesh}).
Seven levels of grid refinement were used in the simulations, for finite element interpolation spaces of polynomial degree ranging from first to fifth order.
A sampling of computed solutions from various grid rotations are included in Figure \ref{fig:visual}.

The results displayed in Figures~\ref{fig:H0results} and ~\ref{fig:H1results} show the convergence rates of the $L^2$-norm and the $H^1$-semi-norm (respectively) for polynomial orders one through five. The mean error rates of all rotations are presented, along with the individual rates for each rotation.
The $L^2$-norm and $H^1$-seminorm of the errors converge with rates $k + 1$ and $k$ (respectively), which are optimal for Lagrange elements of polynomial order $k$.

\begin{figure}[t!]
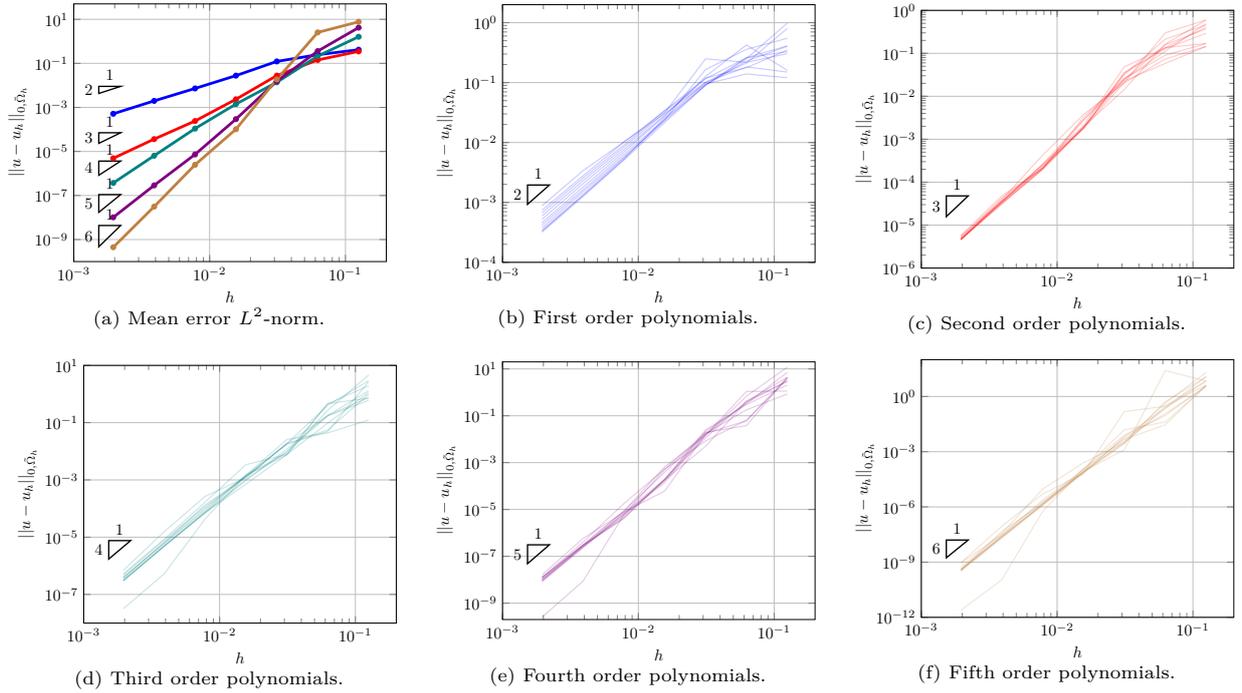

	\begin{subfigure}[tbh!]{0.33\textwidth}

		\vspace{-.2cm}
                \caption{Mean error $L^2$-norm.}
	\end{subfigure}
        \begin{subfigure}[tbh!]{0.33\textwidth}\centering
		%
		\vspace{-.2cm}
                \caption{First order polynomials.}
	\end{subfigure}
        \begin{subfigure}[tbh!]{0.33\textwidth}\centering
		%
		\vspace{-.2cm}
                \caption{Second order polynomials.}
	\end{subfigure}
	\\[.2cm]
        \begin{subfigure}[tbh!]{0.33\textwidth}\centering
		%
		\vspace{-.2cm}
                \caption{Third order polynomials.}
	\end{subfigure}
        \begin{subfigure}[tbh!]{0.33\textwidth}\centering
		%
		\vspace{-.2cm}
                \caption{Fourth order polynomials.}
	\end{subfigure}
        \begin{subfigure}[tbh!]{0.33\textwidth}\centering
		%
		\vspace{-.2cm}
                \caption{Fifth order polynomials.}
	\end{subfigure}
        \caption{Poisson problem with Dirichlet boundary conditions. Convergence rates of the $L^2$-norm of the error, $|| u - u_h ||_{0,\tO} $, using first, second, third, fourth, and fifth order polynomial elements. The mean error computed over all grid rotations is presented in (a), errors for individual grid rotations are presented in (b)-(f). }
        \label{fig:H0results}
\end{figure}

\begin{figure}[tbh!]
	\begin{subfigure}[tbh!]{0.33\textwidth}
		%
		\vspace{-.2cm}
                \caption{Mean error $H^1$-seminorm.}
	\end{subfigure}
        \begin{subfigure}[tbh!]{0.33\textwidth}\centering
		%
		\vspace{-.2cm}
                \caption{First order polynomials.}
	\end{subfigure}
        \begin{subfigure}[tbh!]{0.33\textwidth}\centering
		%
		\vspace{-.2cm}
                \caption{Second order polynomials.}
	\end{subfigure}\\[.2cm]
        \begin{subfigure}[tbh!]{0.33\textwidth}\centering
		%
		\vspace{-.2cm}
                \caption{Third order polynomials.}
	\end{subfigure}
        \begin{subfigure}[tbh!]{0.33\textwidth}\centering
		%
		\vspace{-.2cm}
                \caption{Fourth order polynomials.}                
	\end{subfigure}
        \begin{subfigure}[tbh!]{0.33\textwidth}\centering
		%
		\vspace{-.2cm}
                \caption{Fifth order polynomials.}
	\end{subfigure}
        \caption{Poisson problem with Dirichlet boundary conditions. Convergence rates of the $H^1$-seminorm of the error, $ | u - u_h |_{1,\tO} $, using first, second, third, fourth, and fifth order polynomial elements. The mean error over all grid rotations is presented in (a), errors for individual grid rotations are presented in (b)-(f). 
        }
        \label{fig:H1results}
\end{figure}

Additionally, we also computed the condition numbers $ \kappa(\mathbf{A})$ of the matrix $\mathbf{A}$ associated with~\eqref{eq:UnsymmetricShiftedNitscheBilinearForm} and of entries $\mathrm{A}_{IJ}= a^k_h(\mathrm{N}_J \, , \, \mathrm{N}_I)$.
Here $\mathrm{N}_J$ represents the global finite element shape function associated with the degree-of-freedom (node) $J$.
The mean values of $ \kappa(\mathbf{A})$, shown in Figure~\ref{fig:condnum}, scale with $h^{-2}$ as predicted in Section~\ref{sec:cond_numbers}.
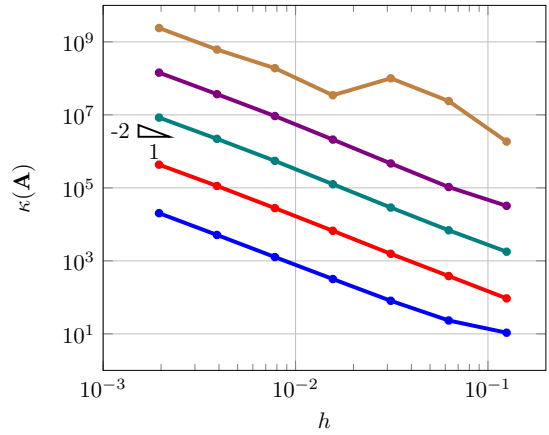
\begin{figure}[tbh!]
  \centering
  \begin{tikzpicture}[scale= 0.85]
    \begin{loglogaxis}[
	grid=major,x post scale=1.0,y post scale =1.0,
	xlabel={$h$},ylabel={$ \kappa(\mathbf{A})$},xmax=2E-01, xmin=1E-03, ymax=1E10, ymin=1E0,
	legend style={font=\footnotesize,legend cell align=left},
	legend pos= south east]
      \addplot[ mark=*,blue, opacity=1,line width=1.75pt,mark size=1pt] coordinates {
        (1.25E-01 , 1.07608932e+01)
        (6.25E-02 , 2.32911133e+01 )
        (3.125E-02 , 8.06784657e+01 )
        (1.5625E-02 , 3.17977344e+02 )
        (7.8125E-03 , 1.27302603e+03 )
        (3.90625E-03 , 5.13341752e+03 )
        (1.953125E-03 , 2.02728927e+04 )
      };
      \addplot[ mark=*,red, opacity=1,line width=1.75pt,mark size=1pt] coordinates {
        (1.25E-01 , 9.38123924e+01 )
        (6.25E-02 , 3.84341419e+02 )
        (3.125E-02 , 1.56169059e+03 )
        (1.5625E-02 , 6.66668376e+03 )
        (7.8125E-03 , 2.79048801e+04 )
        (3.90625E-03 , 1.12862446e+05 )
        (1.953125E-03 , 4.32465079e+05 )
      };
      \addplot[ mark=*,teal, opacity=1,line width=1.75pt,mark size=1pt] coordinates {
        (1.25E-01 , 1.78219324e+03 )
        (6.25E-02 , 6.89780375e+03 )
        (3.125E-02 , 2.89047611e+04 )
        (1.5625E-02 , 1.26047668e+05 )
        (7.8125E-03 , 5.49312946e+05 )
        (3.90625E-03 , 2.20186439e+06 )
        (1.953125E-03 , 8.42422223e+06 )
      };
      \addplot[ mark=*,violet, opacity=1,line width=1.75pt,mark size=1pt] coordinates {
        (1.25E-01 , 3.21879050e+04 )
        (6.25E-02 , 1.05403954e+05 )
        (3.125E-02 ,4.66956318e+05 )
        (1.5625E-02 , 2.10450174e+06 )
        (7.8125E-03 , 9.27974663e+06 )
        (3.90625E-03 , 3.69486746e+07 )
        (1.953125E-03 , 1.43682920e+08 )
      };
      \addplot[ mark=*,brown, opacity=1,line width=1.75pt,mark size=1pt] coordinates {
        (1.25E-01 , 1.85657777e+06 )
        (6.25E-02 , 2.39474717e+07 )
        (3.125E-02 , 1.00577770e+08 )
        (1.5625E-02 , 3.44657107e+07 )
        (7.8125E-03 , 1.90223797e+08 )
        (3.90625E-03 , 6.15761975e+08 )
        (1.953125E-03 , 2.39310760e+09 )
      };
      \ReverseLogLogSlopeTriangleFlip{0.15}{0.07}{0.64}{-2}{black}{0.3};
    \end{loglogaxis}
  \end{tikzpicture}
  \vspace{-.2cm}
  \caption{Poisson problem with Dirichlet boundary conditions. Mean condition numbers for interpolation spaces ranging from first order polynomials (blue) to fifth order polynomials (brown).}
  \label{fig:condnum}
\end{figure}

\subsection{Linear Elasticity with Dirichlet and Neumann boundary conditions}

Using the same geometric setup and grids of the previous test, we also considered the simulation of isotropic compressible linear elasticity equations, with mixed boundary conditions.
The Young's modulus is $E = 10 GPa$, and the Poisson's ratio is $\nu = 0.3$.
The Neumann boundary $\G_N$ is located on the top of the square, while the remaining three sides (left, right, and bottom) constitute the Dirichlet boundary $\G_D$.

As before, seven levels of grid refinement were used in the simulations, for finite element interpolation spaces of polynomial degree ranging from first to fifth order.
Similarly, the computational grids were incrementally rotated from 0 degrees (see Fig.~\ref{fig:mesh}) to 45 degrees (see Fig.~\ref{fig:rotmesh}).

The analytical solution was chosen to be
{\begin{equation}
\begin{cases}
u_{x}(x,y) = 10 \, \pi \sin(10\pi x) \sin(10\pi y)\; , \\
u_{y}(x,y) =  10\, \pi \cos(10\pi x) \cos(10\pi y) \; .
\end{cases}
\end{equation}
and a corresponding forcing function $\bs{b}(x,y)$ was applied with the method of manufactured solutions.

\begin{figure}[tbh!]
	\begin{subfigure}[tbh!]{0.33\textwidth}

		\vspace{-.2cm}
                \caption{Mean error $L^2$-norm.}
	\end{subfigure}
        \begin{subfigure}[tbh!]{0.33\textwidth}\centering
		%
		\vspace{-.2cm}
                \caption{First order polynomials.}
	\end{subfigure}
        \begin{subfigure}[tbh!]{0.33\textwidth}\centering
		%
		\vspace{-.2cm}
                \caption{Second order polynomials.}
	\end{subfigure}
	\\[.2cm]
        \begin{subfigure}[tbh!]{0.33\textwidth}\centering
		%
		\vspace{-.2cm}
                \caption{Third order polynomials.}
	\end{subfigure}
        \begin{subfigure}[tbh!]{0.33\textwidth}\centering
		%
		\vspace{-.2cm}
                \caption{Fourth order polynomials.}
	\end{subfigure}
        \begin{subfigure}[tbh!]{0.33\textwidth}\centering
		%
		\vspace{-.2cm}
                \caption{Fifth order polynomials.}
	\end{subfigure}
        \caption{Linear elasticity problem with Dirichlet and Neumann boundary conditions: Convergence rates for the $L^2$-norm of the error,  $|| \bs{u} - \bs{u}_h ||_{0,\tO} $, using first, second, third, fourth, and fifth order polynomial elements. The mean error computed over all grid rotations is presented in (a), errors for individual grid rotations are presented in (b)-(f). }
        \label{fig:H0resultsLE}
\end{figure}

\begin{figure}[tbh!]
	\begin{subfigure}[tbh!]{0.33\textwidth}
		%
		\vspace{-.2cm}
                \caption{Mean error $H^1$-norm.}
	\end{subfigure}
        \begin{subfigure}[tbh!]{0.33\textwidth}\centering
		%
		\vspace{-.2cm}
                \caption{First order polynomials.}
	\end{subfigure}
        \begin{subfigure}[tbh!]{0.33\textwidth}\centering
		%
		\vspace{-.2cm}
                \caption{Second order polynomials.}
	\end{subfigure}\\[.2cm]
        \begin{subfigure}[tbh!]{0.33\textwidth}\centering
		%
		\vspace{-.2cm}
                \caption{Third order polynomials.}
	\end{subfigure}
        \begin{subfigure}[tbh!]{0.33\textwidth}\centering
		%
		\vspace{-.2cm}
                \caption{Fourth order polynomials.}                
	\end{subfigure}
        \begin{subfigure}[tbh!]{0.33\textwidth}\centering
		%
		\vspace{-.2cm}
                \caption{Fifth order polynomials.}
	\end{subfigure}
        \caption{Linear elasticity problem with Dirichlet and Neumann boundary conditions: Convergence rates for the $H^1$-seminorm of the error, $ \, | \bs{u} - \bs{u}_h |_{1,\tO} $, using first, second, third, fourth, and fifth order polynomial elements. The mean error over all grid rotations is presented in (a), errors for individual grid rotations are presented in (b)-(f).}
        \label{fig:H1resultsLE}
\end{figure}

We can see from Figure~\eqref{fig:H0resultsLE} that the $L^2$-norm of the solution error has a suboptimal rate of convergence by one order, that is polynomial approximations of order $p$ converge with order $p$ and not $p+1$. 
On the other hand, as seen in Figure~\ref{fig:H1resultsLE}, optimal orders of convergence ($p$) are recovered in the $H^1$-seminorm.
This results were expected, and show that the proposed penalty-free version of the SBM only loses one order of convergence in spite of the fact that the surrogate boundary $\tG$ is represented by straight segments.

\section{Summary \label{sec:summary} }
We have proposed a penalty-free version of the SBM, which avoids the tedious selection/estimation of a penalty parameter.
The analysis of stability and convergence demonstrates that the method is stable and convergent for any order of accuracy, and these theoretical results are confirmed by a series of numerical experiments. In the course of the numerical analysis, we also discovered an important interpretation of the SBM as a Galerkin consistent method, a new and somewhat unexpected result on the SBM.

\section*{Acknowledgments}
J. Haydel Collins and Guglielmo Scovazzi have been  supported by the National Science Foundation, Division of Mathematical Sciences (DMS), under Grant 2207164. 

\section*{References}
\bibliographystyle{plain}
\bibliography{./SBM_bibliography} 

\end{document}